\documentclass{amsart}
\usepackage{amssymb}
\usepackage{amsmath}
\usepackage{subfigure}
\usepackage{epsfig}
\usepackage[all]{xy}

\newtheorem{theorem}{Theorem}[section] 
\newtheorem{lemma}[theorem]{Lemma}     

\newtheorem{proposition}[theorem]{Proposition}

\theoremstyle{definition}
\newtheorem{example}[theorem]{Example}
\newtheorem{remark}[theorem]{Remark}

\title[On genera containing split Eichler orders]
{On genera containing non-split Eichler orders over function fields}

\author{Luis Arenas-Carmona}

\author{Claudio Bravo}

\newcommand\ad{\mathbb{A}}

\newcommand\alge{\mathfrak{A}}

\newcommand\oink{\mathcal O}

\newcommand\enteri{\mathbb Z}

\newcommand\Da{\mathfrak{D}}
\newcommand\Ra{\mathfrak{R}}
\newcommand\Ha{\mathfrak{H}}
\newcommand\Ea{\mathfrak{E}}
\newcommand\Fa{\mathfrak{F}}

\newcommand\bmattrix[4]{\left(\begin{array}{cc}#1&#2\\#3&#4\end{array}\right)}
\newcommand\sbmattrix[4]{\textnormal{\scriptsize$\left(\begin{array}{cc}#1&#2\\#3&#4\end{array}\right)$\normalsize}}

\newcommand\matrici{\mathbb{M}}
\newcommand\finitum{\mathbb{F}}

\begin{document}
\maketitle

\footnotesize
\textbf{MSC Numbers (2010): 14H60-11R58 (primary), 14G15-20E08 (secondary)}

\normalsize

\begin{abstract}
Grothendieck-Birkhoff Theorem states that every finite dimensional vector bundle over the projective line
$\mathbb{P}^1$ splits as the 
sum of one dimensional vector bundles. This can be rephrased, in terms of orders, as stating that all
maximal $\mathbb{P}^1$-orders in a matrix algebra split. 
 In this work we study the extent to which this result can be generalized to Eichler
$\mathbb{P}^1$-orders when the base field $\mathbb{F}$ is finite. 
To be precise, we  caracterize both the genera of Eichler orders containing only split orders
and the genera containing only a finite number of non-split conjugacy clases. The latter
characterization is given for arbitrary projective curves over $\mathbb{F}$.
The method developped here also allows us to compute quotient graphs for some
subgroups of $\mathrm{PGL}_2(\mathbb{F}[t])$ of arithmetical interest.
\end{abstract}

\section{Introduction}\label{intro}

Split orders in the $4$-dimensional matrix algebra $\mathbb{M}_2(k)$,  where $k$ is a local
field, were characterized by Hijikata in \cite{Hijikata}. By definition, an order in $\mathbb{M}_2(k)$
 is split if it contains an isomorphic copy of the ring $\oink_k\times\oink_k$, where $\oink_k$ is the ring 
of integers in $k$, or equivalently, if it has the form
$$\bmattrix{\oink_k}IJ{\oink_k},$$
where $I$ and $J$ are fractional ideals. Hijikata proved
 these to be either maximal orders or intersections of
two maximal orders. These are local properties, and in fact, 
for any global field $K$, and for any ring $\oink_S\subseteq K$ of 
$S$-integers, i.e., elements that are integral outside a nonempty finite set $S$ of places that includes the
archimedean places if any, global split $\oink_S$-orders in  $\mathbb{M}_2(K)$ share the same 
characterization.

 When $K$ is a global function field, i.e., the field of rational functions on a smooth irreducible
projective curve $X$ over a finite field $\finitum$, we define $X$-orders in 
$\mathbb{M}_2(K)$ as sheaves of rings
whose generic fiber is $\mathbb{M}_2(K)$ \cite{rrtflpo}. This is usually regarded as the case $S=\emptyset$
in the theory of orders, and this point of view has been fruitful in the past to study quotients
of Bruhat-Tits trees by groups of arithmetical interest (c.f. \cite{cqqgvro}). 
The preceding characterization fails in this setting, as one would expect, giving the absence 
of a Strong Aproximation Theorem with respect to the empty set. However, we do have a result
in this direction, although a significantly more specific one. This is essentially Grothendieck-Birkhoff Theorem
\cite[Thm. 2.1]{burban}, which implies, as we see below, the following statement: 
\begin{quote}
 \textbf{Theorem GB:}
\emph{Every maximal $X$-order  in $\mathbb{M}_2(K)$ is split
when $X$ is the projective line $\mathbb{P}^1$.}
\end{quote}

There is a also a finiteness result that can be regarded as a partial generalization of the preceding statement 
to an arbitrary smooth projective curve defined over a finite field. It follows from the description of the
classifying graph in \cite{cqqgvro} (c.f. \S3):  
\begin{quote} {\bf Finiteness Theorem:} \emph{If $X$ is an arbitrary smooth projective curve over a finite
field, all but finitely many isomorphism classes of maximal $X$-orders  in $\mathbb{M}_2(K)$
contain only split orders.} \end{quote}
The purpose of the present work is to study the
extent to which these results extend to Eichler orders, i.e., intersections of two maximal orders.
The theory introduced here to prove these results can be used to compute quotient graphs of arithmetical interest,
as we exemplify in the last section of this paper.

We start by recalling some basic facts on bundles and lattices.
Let $\oink_X$ denote the structure sheaf of the curve $X$. We can assume that
$\finitum$ equals the full constant field $\oink_X(X)$ of $K$, as we do in the sequel.
An $X$-lattice 
$\Lambda$ is a locally free sheaf of  $\oink_X$-modules of finite rank $n$.
 The group of global sections $\Lambda(X)$ is a finite dimensional vector space over $\finitum$ for any 
$X$-lattice $\Lambda$. 
The sheaf of sections of a vector bundle is an $X$-lattice, and as usual we identify the bundle with the
corresponding  lattice. 
The generic fiber $\Lambda\otimes_{\oink_X}K$ is isomorphic to $K^n$ as a vector space over $K$,
 and we fix one such isomorphism by saying that $\Lambda$ is a lattice in $K^n$. Equivalently, we choose
a $K$-linearly independent set of $n$  sections over some afine subset $U_0\subset X$ and identify it with the canonical basis
of $K^n$. 
This implies that the group of $U$-sections $\Lambda(U)$ is identified with a subset of $K^n$ for any
open set $U\subseteq X$. 
  Thus defined, two lattices $\Lambda$ and $\Lambda'$, or their corresponding bundles,
are isomorphic if and only if there exists an invertible n-by-n matrix $T\in\mathrm{GL}_n(K)$
 satisfying $T\Lambda=\Lambda'$.  
 Similar conventions applies to other explicit vector spaces.
 Note that $\Lambda(U)$ is a lattice over the Dedekind domain $\oink_X(U)$
as defined in \cite{Om}.
An order $\Ra$ in a $K$-algebra $\alge$ is an $X$-lattice in $\alge$ such that $\Ra(U)$ is a ring for any
open subset $U$, e.g., the structure sheaf $\oink_X$ is an $X$-order in $K$. 
We let $\Ra$, $\Da$ and $\Ea$ denote 
$X$-orders of maximal rank in $\mathbb{M}_2(K)$ in all that follows.

 Recall that every $X$-bundle in the one dimensional space $K$  has the form 
$$\mathfrak{L}^B(U)=\left\{f\in K\Big|\mathrm{div}(f)|_U+B|_U\geq0\right\},$$
for some fixed divisor $B$ on $X$, and for every open set $U\subseteq X$.
These bundles are usually called invertible bundles in current literature, and they
have the following properties:
\begin{enumerate}
\item  Linearly equivalent divisors define isomorphic bundles,
\item $\mathfrak{L}^B\mathfrak{L}^D= \mathfrak{L}^{B+D}$, for every pair of divisors $(B,D)$,
\item $\mathfrak{L}^B(U)\subseteq\mathfrak{L}^D(U)$ for every open set $U$ if and only if
$B\leq D$ and
\item $\mathfrak{L}^{\mathrm{div}(g)}=g^{-1}\oink_X$.
\end{enumerate}
In (2), $\mathfrak{L}^B\mathfrak{L}^D$ denotes the sheaf defined by $(\mathfrak{L}^B\mathfrak{L}^D)(U)=
\mathfrak{L}^B(U)\mathfrak{L}^D(U)$ on open sets $U\subseteq X$, which is isomorphic to the
tensor product $\mathfrak{L}^B\otimes_{\oink_X}\mathfrak{L}^D$. In higher dimensions, similar conventions apply
 to scalar products or other bilinear maps. 

A split $X$-lattice or split $X$-bundle is a lattice isomorphic to a direct sum of invertible bundles, 
e.g., a two dimensional $X$-lattice $\Lambda$ is split if $\Lambda\cong\mathfrak{L}_1\times\mathfrak{L}_2$,
as $\oink_X$-modules, where $\mathfrak{L}_1$ and $\mathfrak{L}_2$ are  invertible bundles.
We say that a basis $\{e_1,e_2\}$ splits or diagonalizes an $X$-bundle $\Lambda$ in $K^2$ if 
$\Lambda=\mathfrak{L}_1e_1\oplus\mathfrak{L}_2e_2$, where $\mathfrak{L}_1$ and $\mathfrak{L}_2$
are invertible bundles. Certainly a bundle in $K^2$ is split if and only if  it is split by at least one basis.

 To each $X$-bundle $\Lambda$ in $K^2$ we associate
the order $\Da_{\Lambda}=\mathcal{E}\mathit{nd}_{\oink_X}(\Lambda)$
 in the matrix algebra $\matrici_2(K)$, which can be defined by 
$$\Da_\Lambda(U)=\left\{a\in \matrici_2(K)\Big|a\Lambda(U)\subseteq\Lambda(U)\right\},$$
for every open set $U\subseteq X$. This is a maximal order in $\mathbb{M}_2(K)$ and every maximal order
of this algebra has the form $\Da_\Lambda$ for some $X$-bundle $\Lambda$ in $K^2$.  
 The $X$-bundle $\Lambda$ is split by a certain basis $\{e_1,e_2\}$
if and only if the corresponding maximal order has the form 
$\Da_\Lambda=\sbmattrix {\oink_X}{\mathfrak{L}^{-D}}{\mathfrak{L}^{D}}{\oink_X}$, for
some divisor $D$, in that basis. In fact, if $\Lambda=\mathfrak{L}^Be_1\oplus\mathfrak{L}^Ce_2$,
we have $D=C-B$. This condition on $\Da_\Lambda$
 is equivalent to $\sbmattrix 1000, \sbmattrix 0001\in\Da_\Lambda(X)$.
More generally, we say that an order $\Ea$ is split if 
$\Ea$ is conjugate to $\sbmattrix {\oink_X}{\mathfrak{L}_1}{\mathfrak{L}_2}{\oink_X}$ for some pair of
invertible bundles $(\mathfrak{L}_1,\mathfrak{L}_2)$, or equivalently,
 if its ring of global sections contain a non-trivial idempotent. A split order is split as a lattice but
the converse is false in general.

\begin{quote}
 \textbf{Grothendieck-Birkhoff Theorem \cite[Thm. 2.1]{burban}:}
 Every bundle over $X=\mathbb{P}^1$ is a product
of one dimensional bundles.
\end{quote}

 It is well known that two vector bundles $\Lambda$ and $M$ in $K^2$
satisfy $\Da_\Lambda=\Da_M$ if and only if there exists an invertible  vector bundle 
$\mathfrak{L}$ such that $\Lambda=\mathfrak{L}M$, where the product on the right is 
the scalar product in the vector space $K^2$.
It is apparent that every basis splitting $M$ splits also $\mathfrak{L}M$, so the splittings
of a certain bundle can be more naturally studied in terms of the corresponding maximal order.
In this context, Theorem GB is a particular case of Grothendieck-Birkhoff Theorem.

For every pair of maximal orders  $\Da_\Lambda$ and $\Da_{\Lambda'}$, we consider the Eichler order
$\Ea_{\Lambda,\Lambda'}=\Da_\Lambda\cap\Da_{\Lambda'}$. This is an order of maximal rank in $\matrici_2(K)$.
It follows easily from Hijikata's local characterization that split orders are Eichler, as being Eichler
is a local property, but the converse
is not always true. It follows from the results in this work that non-split Eichler orders exists for every curve 
$X$. This is hardly surprising for geometry experts, as splitting bundles are a thin subset of the moduli
space for curves of higher genus.

\begin{example}
A consequence of Hijikata's characterization of local split orders is the following:
For every pair of lattices $\Lambda$ and $\Lambda'$ in $k^2$, there exists a basis $\{e_1,e_2\}$
 for which $\Lambda=I_1e_1\oplus I_2e_2$ and $\Lambda'=J_1e_1\oplus J_2e_2$, for suitable
ideals $I_1,I_2,J_1,J_2\subseteq\oink_k$. In other words, there is a basis splitting both lattices 
simultaneously. This also holds for arbitrary Dedekind domains, and it is the fundation
of the theory of invariant factors for lattices (c.f. \cite{Om}). Similarly, in the present context,
characterizing split Eichler orders solves the problem of determining whether there is a common
basis splitting two given lattices in $K^2$, or equivalently, whether a common change of variables 
can take a pair of vector bundles into a split form simultaneously.
\end{example}

As we recall in \S2 bellow, an order of maximal rank in $\matrici_2(K)$, or more generally a lattice $\Lambda$
 in a vector space $V$, is completely determined  by its set of completions 
$\left\{\Lambda_P\subseteq V_P\Big| P\in|X|\right\}$ (c.f. \S2), where $|X|$ denote the set of closed points of $X$.
Such orders are usually classfied into genera. A genus is a maximal set of locally isomorphic  orders.
Equivalently, two orders are in the same genus if their completions at all local places are conjugate.
Class Field Theory has been used for a time to classify orders in a genus. This theory allows us to
split a genus into spinor genera. A spinor genus, in a given genus, is a maximal subset
whose lattices are isomorphic over
all but a finite number of affine subsets of $X$. We recall part of this theory in \S2, where
a more technical, but equivalent, definition of spinor genus is given. For a full account,
we refer the reader to \cite{abelianos}. Orders in a spinor genus are classified via quotient graphs.
We recall this theory in \S3, but we refer the reader to \cite{cqqgvro} for a full account on this subject.

A full description of the relation between the spinor genus of an Eichler order and those of the maximal
orders containing it is given in \cite[\S6]{scffgeo}. We just need to recall, for our purposes, that
the genus of an Eichler order $\Ea$ is determined by its level. At a local place $P$,
 the level is the natural distance, in the Bruhat-Tits tree (c.f. \S3), 
between the unique pair of maximal orders whose intersection is the completion $\Ea_P$ (c.f. \S2).
In the global context, the level of an Eichler order
 $\Ea_{\Lambda,\Lambda'}$ is an efective divisor $D=D(\Da_\Lambda,\Da_{\Lambda'})$
defined in terms of these local distances (c.f. \S2). It can also be characterized by the following
property:
\begin{quote}
For every affine open set $U\subseteq X$, we have an isomorphism of $\oink_X(U)$-modules
$$\Da_\Lambda(U)/\Ea_{\Lambda,\Lambda'}(U)\cong
\Da_{\Lambda'}(U)/\Ea_{\Lambda,\Lambda'}(U) \cong\oink_X(U)/\mathfrak{L}^{-D}(U).$$
\end{quote}
In terms of this distance, our main results are as follows:

\begin{theorem}\label{t1}
For an arbitrary smooth projective curve $X$ over a finite field, and for any effective divisor $D$, 
there is only a finite number of conjugacy classes of non-split Eichler orders of level $D$
 if and only if $D$ is multiplicity free, i.e., $D=\sum_{i=1}^nP_i$, where $P_1,\dots,P_n$
are different closed points.
\end{theorem}

When $X=\mathbb{P}^1$ is the projective line, next result can be considered
 a partial generalization of Grothendick-Birkhoff Theorem.

\begin{theorem}\label{t2}
Assume $X=\mathbb{P}^1$ is the projective line. Then the following statements are equivalent for any 
effective divisor $D$:
\begin{enumerate}
\item Every Eichler orders of level $D$ is split.
 \item $D\leq P_1+P_2$, where $P_1\neq P_2$ and $\deg(P_1)=\deg(P_2)=1$. 
\end{enumerate}
\end{theorem}

The main tool in the sequel is the concept of quotient graph, specifically
quotients of the local Bruhat-Tits tree at some place $P$.
 This idea is due to J.-P. Serre who studied the relation between these
quotients and the structure of the arithmetic groups defining them  \cite[\S II.2]{trees}. These are usually
unit groups of maximal orders, and the corresponding quotient is the S-graph, as defined in \cite{cqqgvro}.
 In fact, Serre himself computed the S-graph when $X=\mathbb{P}^1$ and $P$ is a place of degree
 $4$ or less, using tools from algebraic geometry.
 An elementary proof of Serre's result was given in \cite{Mason1}, and some partial generalizations
appear in \cite{Mason2} and \cite{Mason6}. These quotients have been used
to study non-congruence subgroups of Drinfeld modular groups, see \cite{Mason3} or \cite{Mason4}.
We ourselves in \cite{cqqgvro} gave a recursive formula to compute these graphs for a place 
$P\in|\mathbb{P}^1|$ of arbitrary
degree using the theory of spinor genera, and we introduced there the concept of C-graph
(c.f. \S3), which is
used here for the study of conjugacy classes in a genus.  A closed formula for the S-graph for a
maximal order at any place $P\in|\mathbb{P}^1|$ has been given by 
R. K\H ohl, B. M\H uhlherr and K. Struyve 
in \cite{Kohl}, using a different method involving double cosets for simultaneous 
actions on two local trees. The S-graph has also been computed for places of degree 1 on
 an elliptic curve \cite{takahashi}.  M. Papikian has studied the S-graph when $\matrici_2(K)$ is replaced
by a division algebra \cite{Papikian}. Although the theory only requires the orders to
be maximal at the specific place $P$, as far as we can tell the present work is the first attempt 
to use these graphs to study Eichler orders, or any non-maximal order, over a function field.

\begin{remark}
Hijikata's characterization has been generalized to higher dimensional algebras in the local setting
by Shemanske in \cite{soacpib} via Bruhats-Tits Buildings. Bruhat-Tits trees and
buildings play a significant role in the study of the selectivity problem,
 understanding when a commutative order embeds into all, or just into some, 
of the orders in a particular genus \cite{FriedmannQ}, \cite{lino1}, \cite{lino3}. 
This problem arises naturally from questions regarding spectral properties of hyperbolic varieties
\cite{vigneras2}, \cite{lino2}. \end{remark}

Computing quotient graphs provide important information on the structure of a group $G$.
One way to do this is to provide a fundamental domain for $G$, in some suitable Bruhat-Tits tree. 
We do this in \S5 for some congruence subgroups of 
the general linear group $\mathrm{PGL}_2(\mathbb{F}[t])$. 
To make these ideas precise, we recall that 
$\mathrm{PGL}_2(\mathbb{F}[t])\subseteq\mathrm{PGL}_2(K_\infty)$
acts naturally on the Bruhat-Tits tree for the completion at infinite $K_\infty=\mathbb{F}((t^{-1}))$
of $\mathbb{F}(t)$, interpreted as the Ball-tree described in \cite{omeat}.
The vertices of the Ball-tree are in correspondence with the closed balls in $K_\infty$, while its ends
are the elements in the set of $K_\infty$-points $\mathbb{P}^1(K_\infty)$.

\begin{theorem}\label{t4}
Let $N=(t-\lambda_1)\cdots(t-\lambda_n)$ a square-free polynomial with all its roots in $\mathbb{F}$. 
Let $\mathfrak{s}$ be the smallest subtree containing the ends $0$, $\infty$ and $1/M$, for every
proper monic divisor $M$ of $N$. Then the congruence subgroup
$$\Gamma_N=\left\{\bmattrix abcd\in\mathrm{GL}_2(\mathbb{F}[t])\Bigg|
c\equiv 0\ (\mathrm{mod}\ N)\right\}$$
has a fundamental domain of the form $\mathfrak{s}\cup\mathfrak{f}$ for a finite graph
$\mathfrak{f}$.
\end{theorem}

See \S5 for the precise definition of fundamental domain that we use here.

\section{Completions and spinor genera}\label{ndos}

In this section we review the basic facts about spinor genera and spinor class fields of orders.
See \cite{abelianos} for details.

Let $|X|$ be the set of closed points in the smooth projective curve $X$.
 For every such point $P\in |X|$ we let $K_P$ be the completion at $P$ of the function field $K=K(X)$. 
We denote by $\ad=\ad_X$ the adele ring of $X$, i.e., the subring of elements $a=(a_P)_P\in\prod_{P\in|X|}K_P$
for which all but a finite number of coordinates $a_P$ are integral. For any finite dimensional 
vector space $V$ over $K$ we define its adelization $V_\ad=V\otimes_K\ad\cong\ad^{\mathrm{dim}_KV}$,
 and we give it the adelic topology \cite[\S IV.1]{weil}.  Note that $K_\ad\cong\ad$ canonically.
We identify the ring of $\ad$-linear maps $\mathrm{End}_\ad(V_\ad)$ with the adelization $\big(\mathrm{End}_K(V)\big)_\ad$.
For any $X$-lattice $\Lambda$, the completion $\Lambda_P$ is defined as the closure of $\Lambda(U)$
in $V_P$ for an arbitrary affine open set $U$ containing $P$. This definition is independent of the choice of $U$.
 Note that, for every affine subset $U\subseteq X$,
the $\oink_X(U)$-module $\Lambda(U)$ is an $\oink_X(U)$-lattice as defined in \cite{Om}.
 In this work a lattice always means an $X$-lattice or $X$-bundle as in \S1, while we use 
affine lattice for the classical concept.
 The same observation apply to orders and the notations $\Ra$ and $\Ra(U)$.
Just as in the affine case, $X$-lattices are determined
 by their local completions $\Lambda_P$, where $P$ runs over the set $|X|$,
 in the following sense:
\begin{enumerate}
\item For any two lattices $\Lambda$ and $\Lambda'$ in a vector space $V$, $\Lambda_P=\Lambda'_P$
 for almost all $P$,
\item if $\Lambda_P=\Lambda'_P$ for all $P$, then $\Lambda=\Lambda'$, and 
\item every family $\{\Lambda''(P)\}_P$ of local lattices satisfying  $\Lambda''(P)=\Lambda_P$ for almost all $P$ 
is the family of completions of a global lattice $\Lambda''$ in $V$.
\end{enumerate}
The same results apply to orders. 
We also define the adelization $\Lambda_\ad=\prod_{P\in|X|}\Lambda_P$, which is open and compact as
a subgroup of $V_\ad$. This applies in particular to the ring of integral adeles $\oink_\ad=(\oink_X)_\ad
\subseteq K_\ad=\ad$.
It follows from property (3) above that every open and compact $\oink_\ad$-sub-module of $V_\ad$ is 
the adelization of a lattice.
 For every $X$-lattice $\Lambda$ and every invertible element $a\in \mathrm{End}_\ad(V_\ad)$,
 the adelic image $L=a\Lambda$ is the unique  $X$-lattice satisfying $L_\ad=a\Lambda_\ad$. 
The adelic image $L$ thus defined inherit all local properties of the original $X$-lattice $\Lambda$. 
For instance,  adelic images of orders and maximal orders under conjugation are orders
and maximal orders, respectively. In particular, if we fix a maximal $X$-order $\Da$, all maximal $X$-orders
in $\mathbb{M}_2(K)$
 have the form $\Da'=a\Da a^{-1}$ for $a\in\mathrm{GL}_2(\ad):=\matrici_2(\ad)^*$.
 This conjugation must be interpreted as an adelic image. 
More generally, for any fixed order $\Ra$ of maximal rank, the set of orders of the form
$a\Ra a ^{-1}$, for $a\in \mathrm{GL}_2(\ad)$, is called the genus $\mathrm{gen}(\Ra)$. 
The set of maximal $X$-orders is a genus \cite{Eichler2}.

Locally, there is a well defined distance $d_P$ between maximal orders in $\matrici_2(K_P)$. In fact, 
we have $d_P(\Da_P,\Da'_P)=d$ if, in some basis,  both orders take the form
$$\Da_P=\bmattrix {\oink_P}{\oink_P}{\oink_P}{\oink_P}\quad\textnormal{and}\quad
\Da'_P=\bmattrix {\oink_P}{\pi_P^d\oink_P}{\pi_P^{-d}\oink_P}{\oink_P},$$
where $\pi_P$ is a local uniformizing parameter in $K_P$. 
Intersections of orders can be computed locally, in the sense that $\Da_P\cap\Da'_P=(\Da\cap\Da')_P$
for every pair of orders. We define an Eichler order as
the intersection of two maximal orders. This is certainly a local property. 
The level of a local Eichler order is by definition the distance between the
maximal orders defining it. In the local setting, there is a unique pair of maximal orders whose intersection is a
given Eichler order.  Two local Eichler orders are conjugate if and only if their levels coincide.
We conclude that two global Eichler orders $\Ea$ and $\Ea'$ belong to the same genus
precisely when the local levels coincide at all places. Globally, the distance between two maximal orders
$\Da$ and $\Da'$ is defined as the effective divisor
$$D=D(\Da,\Da')=\sum_{P\in|X|}d_P(\Da_P,\Da'_P)P.$$
If $\Da=\Da_\Lambda$ and $\Da'=\Da_{\Lambda'}$, then  $D=D(\Da_\Lambda,\Da_{\Lambda'})$
is, by definition, the level  $\lambda(\Ea_{\Lambda,\Lambda'})$
 of the Eichler order $\Ea_{\Lambda,\Lambda'}$. Two Eichler order
belong to the same genus if and only if they have the same level.
The genus of Eichler orders of level $D$, for any effective divisor $D$, is denoted $\mathbb{O}_D$.

   Two $X$-orders of maximal rank  $\Ra$ and $\Ra'$ in $\matrici_2(K)$ are in the same spinor genus if 
$\Ra'=a\Ra a^{-1}$
for some element $a=bc$ where $b\in\matrici_2(K)$ and $c$ is an adelic matrix satisfying $\mathrm{det}(c)=1_\ad$. 
We write $\Ra'\in\mathrm{Spin}(\Ra)$ in this case. Equivalently,
 two orders $\Ra$ and $\Ra'$ are in the same spinor genus
if and only if they are in the same genus and 
the rings $\Ra(U)$ and $\Ra'(U)$ are conjugate for every affine open subset $U\subseteq X$
(c.f. Remark \ref{rem21}).
 The set of spinor genera in a genus is described via
the spinor class field, which is defined as the class field corresponding to the group 
$K^*H(\Ra)\subseteq \ad^*=:J_X$, where
$$H(\Ra)=\{\mathrm{det}(a)|a\in\mathrm{GL}_2(\ad),\  a\Ra a^{-1}=\Ra\}.$$
This field depends only on the genus $\mathbb{O}=\mathrm{gen}(\Ra)$ of $\Ra$. 
We denote it $\Sigma=\Sigma(\mathbb{O})$.

Let $t\mapsto [t,\Sigma/K]$ denote the Artin map on ideles. There exists a distance map
$\rho:\mathbb{O}\times\mathbb{O}\rightarrow\mathrm{Gal}\big(\Sigma/K\big)$,
defined by $\rho(\Ra,\Ra')=[\det(a),\Sigma/K]$, for any adelic element $a\in\mathrm{GL}_2(\ad)$
satisfying $\Ra'=a\Ra a^{-1}$. This distance map has the following properties:
\begin{enumerate}
\item $\rho(\Ra,\Ra'')=\rho(\Ra,\Ra')\rho(\Ra',\Ra'')$ for any triplet $(\Ra,\Ra',\Ra'')\in\mathbb{O}^3$, and
\item $\rho(\Ra,\Ra')=\mathrm{Id}_{\Sigma(\mathbb{O})}$ if and only if  $\Ra'\in\mathrm{Spin}(\Ra)$.
\end{enumerate}
In particular, for the genus $\mathbb{O}_0$ of maximal orders, the corresponding
distance $\rho_0:\mathbb{O}_0^2\rightarrow\mathrm{Gal}\big(\Sigma_0/K\big)$,
where $\Sigma_0=\Sigma(\mathbb{O}_0)$, is related to the divisor valued distance by the formula
$\rho_0(\Da,\Da')=[[D(\Da,\Da'),\Sigma_0/K]]$,
where $D\mapsto [[D,\Sigma_0/K]]$ is the Artin map on divisors.
 Note however that the distance $\rho_0$ is trivial between isomorphic orders,
which does not hold for  the divisor valued distance.

The spinor class field $\Sigma(D)=\Sigma(\mathbb{O}_D)$, for Eichler orders of level $D=\sum_Pa_PP$,
 is the maximal subfield of $\Sigma_0$ splitting at every place $P$ for which
$a_P$ is odd. The corresponding distance $\rho_D$ is given by the formula
$$\rho_D(\Ea_{\Lambda,\Lambda'},\Ea_{L,L'})=\rho_0(\Da_\Lambda,\Da_L)\Big|_{\Sigma(D)}.$$
The preceding formula follows from \cite[Prop. 6.1]{scffgeo} and the discussion thereafter.

\begin{remark}\label{rem21}
When $\matrici_2(K)$ is replaced by another quaternion algebra $\alge$, the condition for two orders to be
in the same spinor genus goes as follows:
\emph{The orders $\Ra$ and $\Ra'$ are in the same spinor genus if and only if
$\Ra(U)$ and $\Ra'(U)$ are conjugate for any open set $U$ whose complement has at least one place
splitting $\alge$} (c.f. \cite[\S2]{abelianos}). For a matrix algebra, this is equivalent to $U\neq X$.
\end{remark}

\section{Eichler orders and trees}\label{ebundles}

In all of this work, a graph $\mathfrak{g}$ is a pair of sets $V=V(\mathfrak{g})$ and $E=E(\mathfrak{g})$,
 called the vertex set and the edge set, together
with three functions $s,t:E\rightarrow V$ and $r:E\rightarrow E$, 
called respectively source, target and reverse,
satisfying the identities
$$r(a)\neq a,\quad r\big(r(a)\big)=a\textnormal{ and }s\big(r(a)\big)=t(a)$$ for every edge $a$. 
A simplicial map $\gamma:\mathfrak{g}\rightarrow \mathfrak{g}'$ between graphs is a pair of functions
$\gamma_V:V(\mathfrak{g})\rightarrow V(\mathfrak{g}')$ and 
$\gamma_E:E(\mathfrak{g})\rightarrow E(\mathfrak{g}')$
 preserving these functions, and a similar convention applies
to group actions. A group $\Gamma$ acts
on a graph $\mathfrak{g}$ without inversions if
 $g.a\neq r(a)$ for every edge $a$ and every element $g\in\Gamma$. An action without inversions defines
a quotient graph. As mentioned in \S1,
 Basse-Serre Theory allows us to determine the structure of the group $\Gamma$ if we 
understand the quotient graph and the stabilizer of each vertex or edge,
 see \cite[\S I.5]{trees} for an account on this subject. If the action has inversions, we can still 
define a quotient graph by replacing $\mathfrak{g}$ by its barycentric subdivision and ignoring the 
new vertices unless their images 
in the quotient have valency one, in which case
 they are called nonvertices\footnote{we use the term "virtual endpoint" in some of our previous work,
but the use of the word "endpoint" seems to be confusing for some readers.},
 see \cite[Remark 1.6]{cqqgvro} or \cite[Remark 3.1]{rouidqo} for details.
The edge joining a vertex and a nonvertex is called a half-edge. It can be interpreted as an edge that
has been "folded in half" by an inversion. 

The real-line graph $\mathfrak{r}$ is defined by a collection of vertices $\{n_j|j\in\enteri\}$ and a collection of
edges $\{a_j,r(a_j)|j\in\enteri\}$ satisfying both $s(a_j)=n_j$ and $t(a_j)=n_{j+1}$.  An integral interval is
a connected subgraph of $\mathfrak{r}$. A finite integral interval $\mathfrak{i}_{k,k'}$ is completely determined by
its first vertex $n_k$ and it last vertex $n_{k'}$. Its length is $k'-k$.
The notations $\mathfrak{i}_{-\infty,k}$,  $\mathfrak{i}_{k,\infty}$ and  
$\mathfrak{i}_{-\infty,\infty}=\mathfrak{r}$  are defined analogously.
In general, we identify a simplicial map
 $\gamma:\mathfrak{i}_{k,k'}\rightarrow G$ with any shift, i.e., any map
 $\gamma_t:\mathfrak{i}_{k+t,k'+t}\rightarrow G$ satisfying $(\gamma_t)_E(a_{r+t})=\gamma_E(a_r)$.
The reverse of a simplicial map $\gamma:\mathfrak{I}_{0,2}\rightarrow G$  is the map 
$\gamma':\mathfrak{i}_{0,2}\rightarrow G$
satisfying $\gamma'_E(a_1)=\gamma_E\Big(r(a_0)\Big)$ and $\gamma'_E(a_0)=\gamma_E\Big(r(a_1)\Big)$.
This definition generalizes easily to longer intervals.
 A path in a graph $\mathfrak{g}$ is an injective simplicial map 
$\gamma:\mathfrak{i}\rightarrow \mathfrak{g}$,
where $\mathfrak{i}$ is an integral interval. A path is finite of length $k$, or infinite in one or two directions, if
so is the corresponding integral interval. The latter, i.e., a map $\mu:\mathfrak{r}\rightarrow \mathfrak{g}$, is called
a maximal path in the sequel. We also say a ray for a map 
$\rho:\mathfrak{i}_{0,\infty}\rightarrow \mathfrak{g}$.
A line is a pair $\{\gamma,\gamma'\}$ of reverse paths.
 By abuse of notation, we often say
let $\gamma:\mathfrak{i}_{0,k}\rightarrow \mathfrak{g}$ be a line, 
but it must me understood that the reverse $\gamma'$
denotes the same line.

Locally, maximal orders in $\matrici_2(K_P)$, or equivalently
homothety classes of lattices in $K_P^2$, are in correspondence with the vertices of the Bruhat-Tits tree
 $\mathfrak{t}(K_P)$ for $\mathrm{PSL}_2(K_P)$ \cite[\S II.1]{trees}. The vertices corresponding to two maximal orders
are neighbors if and only if their local distance, as defined in \S2, is $1$. In this setting, local Eichler orders 
$ \Ea$ of level $k$ are in correspondence with finite lines
$\gamma:\mathfrak{i}_{0,k}\rightarrow\mathfrak{t}(K_P)$.
 In fact, there is a unique path $\gamma=\gamma(v,w)$ conecting every ordered pair $(v,w)=\Big(\gamma_V(n_0),\gamma_V(n_k)\Big)$ of vertices in the tree.
 If we denote by $\Da_v$ the maximal order corresponding
to the vertex $v$, the Eichler order corresponding to a line $\gamma$ as above is $\Ea_\gamma=
\Da_v\cap\Da_w$. The
orders $\Da_{\gamma_V(n_i)}$, for $0\leq i\leq k$, are precisely the maximal orders containing $\Ea_\gamma$.
In the notations of \cite{scffgeo}, the largest subgraph whose vertices contain an order $\Ha$ is denoted $\mathfrak{S}_0(\Ha)$, and in this setting we have
$\mathfrak{S}_0(\Ea_{\gamma})=\gamma(\mathfrak{i}_{0,k})$.
In particular, the local maximal orders in the expression $\Ea=
\Da\cap\Da'$ are unique. 

For a global Eichler order $\Ea$ of level $D=\lambda(\Ea)=\sum_P\alpha_PP$, 
the set of maximal orders containing $\Ea$ is in correspondence with the set of vertices in the finite grid
$\mathbb{S}(\Ea)=\prod_P\mathfrak{S}_0(\Ea_P)$, where $P$ runs over the set of places at which 
$\alpha_P>0$.
Any vertex $v$ of this grid corresponds to a global maximal order $\Da_v$ containing $\Ea$ and conversely.
For any pair $(v_1,v_2)$ of opposite vertices of this grid, the corresponding maximal orders satisfy
$\Ea=\Da_{v_1}\cap\Da_{v_2}$, and for all these pairs the divisor valued distance defined in \S2 is $D$. 
These grids are seen as sub-complexes of a suitable product of Bruhat-Tits trees.
 Fix an effective divisor $D=\sum_P\alpha_PP$. 
Any grid of the form $\mathbb{S}(\Ea)$ for $\lambda(\Ea)=D$  is called a concrete $D$-grid. 
Note that $\mathrm{PGL}_2(K)$ acts by conjugation on the set of concrete $D$-grids. Orbits
of concrete $D$-grids are called ideal $D$-grids. Next result is immediate from the definitions:

\begin{proposition}
For any efective divisor $D$,
the set of conjugacy classes of Eichler orders of level $D$ in $\mathbb{M}_2(K)$
 is in correspondence with the set of ideal $D$-grids. 
\end{proposition}

If we write $D=D'+\alpha_PP$, where $D'$ is supported away from $P$, any concrete $D$-grid $\mathbb{S}(\Ea)$
is a paralellotope having two concrete $D'$-grids as opposite faces. These are called the $P$-faces of the $D$-grid.
The $P$-faces of an ideal $D$-grid are well defined as ideal grids. This convention is used in all that follows.

Now let $Q\in |X|$ and let $\mathbb{O}$ be a genus of orders of maximal rank that are maximal at $Q$. Let $U$
be the complement of $\{Q\}$ in $X$. Fix an order $\Ra\in \mathbb{O}$, and let $\Psi$ be the set of orders
$\Ra'\in\mathbb{O}$ satisfying $\Ra'(U)=\Ra(U)$. These orders are called the $Q$-variants
of $\Ra$. An order $\Ra'\in\Psi$ is completely determined by the 
local order $\Ra'_Q$, and the set of conjugacy classes of these orders is in correspondence with the vertices
of the classifying graph $\mathfrak{c}_Q(\Ra)=\Gamma\backslash\mathfrak{t}(K_Q)$, where $\mathfrak{t}(K_Q)$ is the local
Bruhat-Tits tree at $Q$, and $\Gamma$ is the stabilizer of $\Ra(U)$ in $\mathrm{PGL}_2(K)$. 
 As orders in the same spinor genera restrict to conjugate orders in every affine subset,
every conjugacy class in a given spinor genus $\mathrm{Spin}(\Ra)$ corresponds 
to a unique vertex in $\mathfrak{c}_Q(\Ra)$. The orders in $\Psi$ belong to either one or two spinor 
genera, according two whether  $[[Q,\Sigma(\mathbb{O})/K]]$ is trivial or not, and in 
the later case the quotient graph is bipartite. The classifying graph $\mathfrak{c}_Q(\mathbb{O})$ is defined
as the disjoint union of the graphs corresponding to all spinor genera or pairs of spinor genera. 
Note that this is a straightforward generalization of the definition in \cite{cqqgvro}.

Two orders $\Ra',\Ra''\in\Psi$ are called $Q$-neighbors if the corresponding vertices $\Ra'_Q$
and $\Ra''_Q$ are neighbors in the Bruhat-Tits tree.
Two $Q$-neigbors have equal completions at each place other that $Q$, so next
result is immediate from the definitions:

\begin{proposition} Let $D$ be an efective divisor supported away from the place $Q$.
The vertices of the classifying graph $\mathfrak{c}_Q(\mathbb{O}_D)$
 are in correspondence with the ideal $D$-grids, while its pairs of mutually reverse
edges are in correspondence with the ideal $(D+Q)$-grids. The endpoints
 of an edge are the vertices
corresponding to the $Q$-faces of the grid corresponding to that edge.  
\end{proposition}

Note that it does not suffice to know the conjugacy class of each vertex in the grid to determine the
conjugacy class of the corresponding Eichler order. For example, this is the reason why classifying
graphs of maximal orders describing in \cite{cqqgvro} have multiple edges.

For any divisor $D=\sum_P\alpha_PP$, its absolute value is defined by $|D|=\sum_P|\alpha_P|P$.

\begin{lemma}
Let $\Ea$ be a split Eichler order of level $D$
that can be written as the intersection of two maximal orders isomorphic
to $\Da_B$ and $\Da_{B'}$. Then there are divisors $B_0$ and $B'_0$ such that:
\begin{enumerate}
\item $B_0$ is linearly equivalent to $B$ or $-B$,
\item $B'_0$ is linearly equivalent to $B'$ or $-B'$ and
\item $|B_0-B'_0|=D$.
\end{enumerate} 
\end{lemma}

\begin{proof}
Recall that the set of local maximal orders containing a given idempotent, say $\sbmattrix 1000$, lie
in a maximal path of the corresponding local tree \cite[Cor. 4.3]{Eichler2}. 
Globally, the set of such orders coinciding with $\mathbb{M}_2(\oink_k)$
 outside some finite set $S$ of places
is in correspondence with an infinite grid whose dimension is the cardinality of $S$. Algebraically, they can be described
as the orders of the form $\Da_B=\sbmattrix {\oink_X}{\mathfrak{L}^B}{\mathfrak{L}^{-B}}{\oink_X}$ where 
$B$ is a divisor supported in $S$. It follows that the Eichler orders containing $\sbmattrix 1000$ 
as a global section are the orders 
of the form $\Ea[B,B']=\sbmattrix {\oink_X}{\mathfrak{L}^{-B'}}{\mathfrak{L}^{-B}}{\oink_X}$, where 
$B+B'$ is an effective divisor. In fact, if we define $$G=\sum_P\mathrm{min}\{\alpha_P,\alpha'_P\}P, \qquad
M=\sum_P\mathrm{max}\{\alpha_P,\alpha'_P\}P,$$
where $B=\sum_P\alpha_PP$ and $B'=\sum_P\alpha'_PP$ ,
we have $\Da_B\cap\Da_{B'}=\Ea[M,-G]$, which is an Eichler order of level $M-G$. We
 note that all pairs $(v'',v''')$ of opposite corners of the grid corresponding to $\Ea[M,-G]$
satisfy $\Da_{v''}=\Da_{B''}$ and $\Da_{v'''}=\Da_{B'''}$, where $|B''-B'''|=M-G$. Now the result follows from
\cite[Prop 4.1]{cqqgvro} and the discussion preceding it.
\end{proof}

\begin{lemma}\label{l34}
Let $P\in |\mathbb{P}^1|$ be a point of degree 2 or larger. Then there exists non-split orders in 
$\mathbb{O}_{P}$.
\end{lemma}

\begin{proof}
Let $K=K(\mathbb{P}^1)=\mathbb{F}(t)$.
The conjugacy classes of maximal orders in $\mathbb{M}_2(K)$ are the classes $[\Da_{nP_1}]$ for $n=0,1,2,\dots$
and $P_1$ a place of degree 1 \cite[\S II.2.3]{trees}. Recall that $P$ is, as a divisor, linearly equivalent to $dP_1$ where 
$d=\mathop{\mathrm{deg}}(P)$.
We need to recall some properties of the classifying graph $\mathfrak{c}_P(\mathbb{O}_0)$ of maximal orders:
\begin{enumerate}
\item For any $n>0$, the vertices $[\Da_{nP_1}],[\Da_{(n+d)P_1}],[\Da_{(n+2d)P_1}],\dots$ 
are consecutive vertices in an infinite ray (c.f. \cite[Th. 1.2]{cqqgvro}).
\item The graph has one connected component if $d$ is odd and two if $d$ is even (c.f. \cite[Th. 1.3]{cqqgvro}).
\end{enumerate} 
When $d>2$, there must exists an edge connecting two orders isomorphic to $\Da_{nP_1}$ and $\Da_{mP_1}$,
where neither $n+m$ nor $n-m$ is divisible by $d$.  In particular, if the corresponding Eichler order were split, 
there should exist  two divisors $B_0$ and $B'_0$ of degrees $\pm n$ and $\pm m$
satisfying $|B_0-B'_0|=P$. This can only mean $B_0-B'_0=\pm P$, which is not possible by degree considerations,
 and the result follows from the preceding lemma. Assume now $d=2$.
\begin{figure}[h]
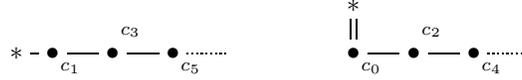

\[  
\xygraph{
!{<0cm,0cm>;<.8cm,0cm>:<0cm,.8cm>::}
!{(0.4,1)}*+{*}="a" 
!{(1,1)}*+{\bullet}="b" !{(1.3,0.7)}*+{{}^{c_1}}="b1"
!{(2,1)}*+{\bullet}="c" !{(2.3,1.3)}*+{{}^{c_3}}="c1"
!{(3,1)}*+{\bullet}="x" !{(3.3,0.7)}*+{{}^{c_5}}="x1"
!{(4,1)}*+{}="d"
!{(6,1.8)}*+{*}="m"
!{(6,1)}*+{\bullet}="e" !{(6.3,0.7)}*+{{}^{c_0}}="e1"
!{(7,1)}*+{\bullet}="f" !{(7.3,1.3)}*+{{}^{c_2}}="f1"
!{(8,1)}*+{\bullet}="g" !{(8.3,0.7)}*+{{}^{c_4}}="g1"
!{(9,1)}*+{}="h"
 "a"-"b" "b"-"c" "c"-"x" "x"-@{.}"d"
 "e"-"f" "f"-"g"
"g"-@{.}"h" "m"-@{=}"e"} 
\]\caption{The two connected components of $\mathfrak{c}_P(\mathbb{O}_0)$ when  
$X=\mathbb{P}^1$ and $\deg(P)=2$. Here $c_m=[\Da_{mP_1}]$.}
\end{figure}
We learn from \cite[Fig. 7]{cqqgvro} that there is an edge in 
$\mathfrak{c}_P(\mathbb{O}_0)$ connecting the class
$[\Da_0]$ to itself, i.e., there is an edge in $\mathfrak{t}(K_P)$ connecting two orders
isomorphic to $\Da_0$. We denoted the corresponding half-edge with double line in Figure 1. 
If the corresponding Eichler order were split, there should exist 
two divisors $B_0$ and $B'_0$ of degree $0$
satisfying $|B_0-B'_0|=P$, and the result follows as before. 
\end{proof}

\begin{remark}
Note that the same argument fails for the half-edge in the other component as
$B_0=P_1$ and $B'_0=P_1-P$ are divisors of degree $\pm1$ satisfying $B_0-B_0'=P$,
and in fact the corresponding Eichler order is split. 
\end{remark}

\begin{remark}\label{r35}
Let $U=X-\{P\}$, let $\Gamma_0=\Gamma_0(\Da)=K^*\Da(U)^*/K^*\subseteq\Gamma$ be the unit group
of $\Da$, and let 
$\mathfrak{s}_P(\Da)=\Gamma_0\backslash\mathfrak{t}(K_P)$ 
be the S-graph of $\Da$ as defined in \cite{cqqgvro}.
 Since $\Gamma_0$ is a normal subgroup of
$\Gamma$, the classifying graph is a quotient of the S-graph,
in the sense that each connented component of the former is a quotient of the latter.
 This can be used as a tool to compute classifying 
graphs, since the valency in the S-graph is easier to compute.
In fact, for any order $\Da'\in\Psi$, the group $\Da'(X)^*$ acts on the 
$\mathbb{F}(P)$-vector space $\Lambda_P/\pi_P\Lambda_P$, where $\Lambda_P$
is the lattice corresponding to the maximal order $\Da'_P$, and $\pi_P$ is a local uniformizing parameter.
This can be interpreted as an action via Moebius transformations on
 the finite projective space $\mathbb{P}^1\big(\mathbb{F}(P)\big)$.
 We identify these orbits with the $P$-neighbors 
of $\Da$. This has a particularly simple description for a split Eichler order $\Da=\Ea[B,B']$:
\begin{quote}
Assume that $B+B'$ is effective and non-zero. Then either $B$ or $B'$ has positive degree, say $B$ to fix ideas.
Then $\mathfrak{L}^{-B}(X)=\{0\}$. A  simple computation shows that
\begin{equation}\label{eqn33}
\Ea[B,B'](X)=\bmattrix {\finitum}{\mathfrak{L}^{-B'}(X)}{0}{\finitum},
\end{equation}
 and any element whose only
eigenvalue is $1$ acts by conjugation as an aditive map of the form $t\mapsto t+a$ on the projective line
$\mathbb{P}^1\big(\finitum(P)\big)$.\end{quote}
 We conclude that any vertex in the
S-graph $\mathfrak{s}_P(\Da)$ corresponding to an order satisfying Equation 
(\ref{eqn33}) has valency $2$ as soon as
$$\dim_{\finitum}\Big(\mathfrak{L}^{-B'}(X)/\mathfrak{L}^{-B'-P}(X)\Big)=[\finitum(P):\finitum]=\deg(P),$$
while its valency is $2+|\finitum(P)^*/\finitum^*|$ if the preceding dimension is $0$.
 It is a consequence of Riemann-Roch' 
Theorem that the valency is always $2$ for large values of $\deg(-B')$.
 In particular, if $P$ is
a point of degree $1$, the valency of a non-maximal split Eichler order can be either $2$ or $3$.
A similar result holds for maximal orders by a slightly refined argument.
\end{remark}

\section{Proof of Theorem \ref{t1} and Theorem \ref{t2} }

We begin this section by proving a few key lemmas. We use throughout the following formulas
\begin{equation}\label{switchends}
\bmattrix 01f0\Ea[B,D]{\bmattrix 01f0}^{-1}=\Ea[D+\mathrm{div}(f),B-\mathrm{div}(f)]
\end{equation}
and
\begin{equation}\label{noswitchends}
\bmattrix f001\Ea[B,D]{\bmattrix f001}^{-1}=\Ea[B-\mathrm{div}(f),D+\mathrm{div}(f)],
\end{equation}
which are proved by a straightforward computation.

\begin{lemma}\label{l41}
Let $P_1,P_2,P_3\in |\mathbb{P}^1|$ be three points of degree 1. Then every order in  $\mathbb{O}_{P_1+P_2}$
is split, but there exists a unique conjugacy class of non-split orders in $\mathbb{O}_{P_1+P_2+P_3}$.
\end{lemma}

\begin{proof}
Recall as before that the local maximal orders containing a fixed non-trivial idempotent,
i.e., a conjugate of  $\sbmattrix 1000$, are the vertices of a maximal path \cite[Cor. 4.3]{Eichler2}.
 On the other hand, the 
classifying graph (or the S-graph) for maximal orders at a point $P_1$ of degree $1$ is as shown in Figure 2A
(c.f. \cite[\S II.2.3]{trees}, or \cite[Fig. 1]{cqqgvro}).  This is covered twice by the maximal path in Figure 2B.
Every edge of this graph corresponds to a conjugacy class of orders in $\mathbb{O}_{P_1}$ and
conversely, whence every order in this genus is split. In fact, all classes in this genus are represented in the
set $$\{\Ea[P_1,0],\Ea[2P_1,-P_1],\Ea[3P_1,-2P_1],\dots\}.$$
Now we draw the classifying graph for $\mathbb{O}_{P_1}$ at a place $P_2\neq P_1$.
It is easy to see that all vertices have valency 2 by the remark at the end of \S3.
Thus we obtain the graph in Figure 2C, where $b_n$ is the class $\big[\Ea[P_1+nP_2,-nP_2]\big]$, which
equals $\big[\Ea[(n+1)P_1,-nP_1]\big]$ by Equation (\ref{noswitchends}).
 The edges in this graph correspond to the classes in  
$\mathbb{O}_{P_1+P_2}$. Again, each of these classes of edges has a representative in
the maximal path corresponding to the global idempotent  $\sbmattrix 1000$. We conclude that each order in this
genus is split. Representatives for all these orders are in the set
$$\{\Ea[P_1,P_2],\Ea[P_1+P_2,0],\Ea[P_1+2P_2,-P_2],\Ea[P_1+3P_2,-2P_2],\dots\}.$$
The first one of these representatives corresponds to the half-edge in Figure 2C. Note that 
each conjugacy classes above can be fully characterized by the conjugacy classes of 
the four maximal orders containing any order in it. For example, the maximal orders containing
the order $\Ea[P_1,P_2]$ have the form $\Da_B$ where $B\leq P_1$ and $-B\leq P_2$,
so $B\in\{ 0,P_1,-P_2,P_1-P_2\}$, and they belong to the classes 
$[\Da_0]$, $[\Da_{P_1}]$, $[\Da_{P_1}]$, and $[\Da_0]$  respectively.

We can iterate this procedure on the classifying graph for $\mathbb{O}_{P_1+P_2}$ (Figure 2D) at a third place
$P_3$, where $d_n=\big[\Ea[P_1+nP_3,P_2-nP_3]\big]=\big[\Ea[nP_2+P_1,(1-n)P_2]\big]$.
\begin{figure}[h]
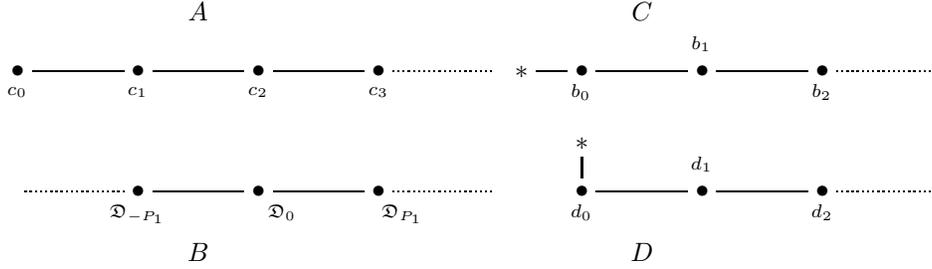

\[  \xygraph{
!{<0cm,0cm>;<.8cm,0cm>:<0cm,.8cm>::}
!{(0,4)}*+{\bullet}="a" !{(0,3.6)}*+{{}^{c_0}}="a1"
!{(2,4)}*+{\bullet}="b" !{(2,3.6)}*+{{}^{c_1}}="b1"
!{(4,4)}*+{\bullet}="c" !{(4,3.6)}*+{{}^{c_2}}="c1"
!{(6,4)}*+{\bullet}="x" !{(6,3.6)}*+{{}^{c_3}}="x1"
!{(8,4)}*+{}="d"
!{(0,2)}*+{}="m"
!{(2,2)}*+{\bullet}="e" !{(2,1.6)}*+{{}^{\Da_{-P_1}}}="e1"
!{(4,2)}*+{\bullet}="f" !{(4.4,1.6)}*+{{}^{\Da_0}}="f1"
!{(6,2)}*+{\bullet}="g" !{(6.4,1.6)}*+{{}^{\Da_{P_1}}}="g1"
!{(8,2)}*+{}="h"
!{(3,5)}*+{A}="z" !{(3,1)}*+{B}="z1"
 "a"-"b" "b"-"c" "c"-"x" "x"-@{.}"d"
 "e"-"f" "f"-"g"
"g"-@{.}"h" "m"-@{.}"e"} 
\xygraph{
!{<0cm,0cm>;<.8cm,0cm>:<0cm,.8cm>::}
!{(1,4)}*+{*}="a" 
!{(2,4)}*+{\bullet}="b" !{(2,3.6)}*+{{}^{b_0}}="b1"
!{(4,4)}*+{\bullet}="c" !{(4,4.4)}*+{{}^{b_1}}="c1"
!{(6,4)}*+{\bullet}="x" !{(6,3.6)}*+{{}^{b_2}}="x1"
!{(8,4)}*+{}="d"
!{(2,2.8)}*+{*}="m"
!{(2,2)}*+{\bullet}="e" !{(2,1.6)}*+{{}^{d_0}}="e1"
!{(4,2)}*+{\bullet}="f" !{(4,2.4)}*+{{}^{d_1}}="f1"
!{(6,2)}*+{\bullet}="g" !{(6,1.6)}*+{{}^{d_2}}="g1"
!{(8,2)}*+{}="h"
!{(3,5)}*+{C}="z" !{(3,1)}*+{D}="z1"
 "a"-"b" "b"-"c" "c"-"x" "x"-@{.}"d"
 "e"-"f" "f"-"g"
"g"-@{.}"h" "m"-"e"} 
\]\caption{Four graphs used in the proof of Lemma 3.4.}\end{figure}
If we try to use this graph to prove that all edges correspond to split orders we find an obstacle.
The image of the vertex $v=\Ea[P_1,P_2]$ has valency $3$ in the S-graph
(c.f. Remark \ref{r35}). Two of its edges $e'$ and $e''$ join it with the images of 
$\Ea[P_1-P_3,P_2+P_3]$ and $\Ea[P_1+P_3,P_2-P_3]$ respectively.
Both latter orders are isomorphic to $\Ea[P_1+P_2,0]$, and we can check that
the images of $e'$ and $e''$ in the classifying graph coincide, as we see by setting 
$\mathrm{div}(f)=P_2-P_1$ in Equation (\ref{switchends}).
Any other edge $e$ whose starting point is $\Ea[P_1,P_2]$ is in the class corresponding
to the third edge in the S-graph. Since every premage, in the S-graph, of the vertex $d_n$, for $n\geq1$,
has valency 2 with non-isomorphic neighbors in the classes $d_{n-1}$ and $d_{n+1}$,
the edge $e$ necesarily joins two orders isomorphic to $\Ea[P_1,P_2]$. We conclude that the classifying 
graph looks like the one in Figure 2D. The vertical half-edge joining $d_0$ with a nonvertex 
has no representative on the main maximal path, but 
it might have a representative in the maximal path corresponding to a different global idempotent.
 We must prove that this is not the case.
Assume that the Eichler order $\Ea$ corresponding to this edge has an idempotent global section $\rho$.
We observe that both $P_3$-faces of the corresponding grid correspond to  
conjugates of the order $\Ea[P_1,P_2]$, whence the maximal orders corresponding to each of the
eight vertices belongs to the class shown in Figure 3.
 Assume a basis is chosen in a way that $\rho=\sbmattrix 1000$. 
Conjugating by a suitable diagonal matrix if needed, we can assume that one of the vertices 
in the class $[\Da_0]$ is actually $\Da_0$. Then, no choice of
the signs in the neighboring vertices, which must be $\Da_{P_i}$ or $\Da_{-P_i}$ in each case, give us
the configuration of classes shown in Figure 3. This is a contradiction.   
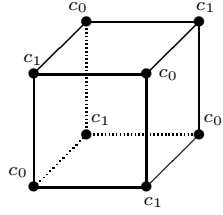
\begin{figure}
\unitlength 1mm 
\linethickness{0.4pt}
\ifx\plotpoint\undefined\newsavebox{\plotpoint}\fi 
\[
\begin{picture}(28,25)(0,0)
\put(0,0){\line(0,1){15}}\put(0,0){\line(1,0){15}}
\put(15,0){\line(0,1){15}}\put(0,15){\line(1,0){15}}
\put(0,15){\line(1,1)7}\put(15,15){\line(1,1)7}
\put(7,22){\line(1,0){15}}\put(22,7){\line(0,1){15}}
\put(0,15){\line(1,1)7}\put(15,15){\line(1,1)7}\put(15,0){\line(1,1)7}
\multiput(0,0)(0.4375,0.4375){16}{\line(0,-1){.05}}
\multiput(7,7)(0,0.5){30}{\line(0,1){.05}}
\multiput(7,7)(0.5,0){30}{\line(1,0){.05}}
\put(0,0){\makebox(0,0)[cc]{$\bullet$}}\put(7,7){\makebox(0,0)[cc]{$\bullet$}}
\put(22,7){\makebox(0,0)[cc]{$\bullet$}}\put(15,0){\makebox(0,0)[cc]{$\bullet$}}
\put(0,15){\makebox(0,0)[cc]{$\bullet$}}\put(7,22){\makebox(0,0)[cc]{$\bullet$}}
\put(22,22){\makebox(0,0)[cc]{$\bullet$}}\put(15,15){\makebox(0,0)[cc]{$\bullet$}}
\put(-2,2){\makebox(0,0)[cc]{${}_{c_0}$}}\put(9,9){\makebox(0,0)[cc]{${}_{c_1}$}}
\put(24,9){\makebox(0,0)[cc]{${}_{c_0}$}}\put(16,-2){\makebox(0,0)[cc]{${}_{c_1}$}}
\put(0,17){\makebox(0,0)[cc]{${}_{c_1}$}}\put(6,24){\makebox(0,0)[cc]{${}_{c_0}$}}
\put(23,24){\makebox(0,0)[cc]{${}_{c_1}$}}\put(18,15){\makebox(0,0)[cc]{${}_{c_0}$}}
\end{picture}
\] \caption{Conjugacy classes $c_n=[\Da_{nP_1}]$
of the maximal orders containing the only non-split Eichler order
in the genus $\mathbb{O}_{B}$, up to conjugacy,
 when $B$ is the sum of three different points of degree 1.}
\end{figure}
\end{proof}

By a cusp, in a graph $\mathfrak{g}$, we mean the image of a ray 
$\gamma:\mathfrak{i}_{0,\infty}\rightarrow\mathfrak{g}$, where $\gamma_V(n_i)$
has valency $2$ for $i\geq1$.
A graph is combinatorially finite if it is the union of a finite graph and a finite number of cusps. Serre
proved in \cite{trees} that the S-graph of a maximal order is combinatorially finite. 
We usually assume that cusps are as big as possible by choosing the valency of
$\gamma_V(n_0)$ different from $2$, whenever possible. This is not the case if
$\mathfrak{g}$ looks like the classifying graph in Figure 2C, where we assume the
initial vertex of the cusp is $\gamma_V(n_0)=b_0$, or when $\mathfrak{g}$ is a maximal path.
In the latter case we choose an arbitrary point as the initial vertex of either cusp.

\begin{example}\label{ex42}
Note that the proceadure applied above to compute the quotient graphs in the preceding proof
can be iterated to describe the classifying graph at $P_\infty$
for every genus of the form $\mathbb{O}_{P_1+\dots+P_n}$
where $P_1,\dots,P_n$ and $P_\infty$ are points of degree 1.
Note that $n\leq|\mathbb{F}|$.  In every step, almost all edges in the cusp
of the previous step become vertices in the new cusp that can be shown to be unique. In fact,
applying equation (\ref{noswitchends}) with $\mathrm{div}(f)=n(P'-P)$ send the edge $e_n$
 in Figure 4 to the edge $f_n$. The square between $f_n$ and $f_{n+1}$
corresponds to an edge in the next step. This proceadure shows that the classifying
graph $C_{P_\infty}(\mathbb{O}_{P_1+\dots+P_n})$ has precisely one cusp.
  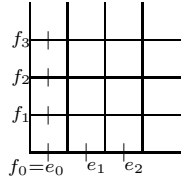
\begin{figure}
\unitlength 1mm 
\linethickness{0.4pt}
\ifx\plotpoint\undefined\newsavebox{\plotpoint}\fi 
\[
\begin{picture}(28,25)(0,0)
\put(0,0){\line(0,1){20}}\put(0,0){\line(1,0){20}}
\put(5,0){\line(0,1){20}}\put(0,5){\line(1,0){20}}
\put(10,0){\line(0,1){20}}\put(0,10){\line(1,0){20}}
\put(15,0){\line(0,1){20}}\put(0,15){\line(1,0){20}}
\put(2.5,0){\makebox(0,0)[cc]{{\scriptsize$|$}}}\put(7.5,0){\makebox(0,0)[cc]{{\scriptsize$|$}}}
\put(12.5,0){\makebox(0,0)[cc]{{\scriptsize$|$}}}\put(2.5,5){\makebox(0,0)[cc]{{\scriptsize$|$}}}
\put(2.5,10){\makebox(0,0)[cc]{{\scriptsize$|$}}}\put(2.5,15){\makebox(0,0)[cc]{{\scriptsize$|$}}}
\put(1,-2){\makebox(0,0)[cc]{${}_{f_0=e_0}$}}
\put(9,-2){\makebox(0,0)[cc]{${}_{e_1}$}}
\put(14,-2){\makebox(0,0)[cc]{${}_{e_2}$}}
\put(-1,5){\makebox(0,0)[cc]{${}_{f_1}$}}
\put(-1,10){\makebox(0,0)[cc]{${}_{f_2}$}}
\put(-1,15){\makebox(0,0)[cc]{${}_{f_3}$}}
\end{picture}
\] \caption{ Horizontal neighbors are $P$-neighbors,
while vertical neighbors are $P'$-neighbors. Relevant edges are marked "$|$".}
\end{figure}
\end{example}

\begin{lemma}\label{l42}
Let $X$ be an arbitrary smooth projective curve, and
let $P\in |X|$ be an arbitrary point. Then there is an infinite
set of conjugacy classes of non-split orders in  $\mathbb{O}_{2P}$.
\end{lemma}

\begin{proof}
Fix an order $\Ea$ of level $2P$ and a maximal order $\Da$ containing $\Ea$. 
 Any cusp in the classifying graph $C_P(\Da)$ looks like the one in Figure 5A, where each
order in the class $[\Da_{B+nP}]$, for $n\geq1$, has one neighbor in the class $[\Da_{B+(n+1)P}]$ and
all the others in the class $[\Da_{B+(n-1)P}]$. Since the orders $\Ea'\in\mathbb{O}_{2P}$
satisfying $\Ea'(U)=\Da(U)$, where $U=X-\{P\}$, correspond to lines of length $2$
 in the Bruhat-Tits tree at $P$, for every
value of $n>1$ there exists an Eichler order contained in one order in the class $[\Da_{B+nP}]$ 
and two orders in the class $[\Da_{B+(n-1)P}]$. We claim that such orders are non-split for $n>-\mathrm{deg}(B)$.
As they are evidently in different conjugacy classes, the result follows from the claim.  Now let $\Ea$
be an Eichler order of level $2P$ whose grid has vertices in the conjugacy classes shown
 in Figure 5B. If $\Ea$ were split, by an appropiate 
choice of coordinates, we can assume $\Ea=\Ea[D,D']$, where $D+D'=2P$ or $-D'=D-2P$, whence
the three maximal orders containing $\Ea$ must be $\Da_D$, $\Da_{D-P}$ and $\Da_{D-2P}$,
with $D-P$ linearly equivalent to $B+nP$, and hence of positive degree. 
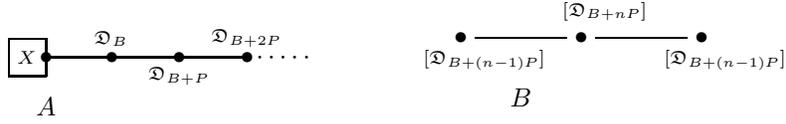
\begin{figure}[h]
\[  
\unitlength 1mm 
\linethickness{0.4pt}
\ifx\plotpoint\undefined\newsavebox{\plotpoint}\fi 
\begin{picture}(40,15)(0,-6)
\put(0,0){\line(0,1){5}}\put(0,5){\line(1,0){5}}\put(0,0){\line(1,0){5}}\put(5,0){\line(0,1){5}}
\put(5,2.5){\makebox(0,0)[cc]{$\bullet$}}\put(5.2,2.5){\line(1,0){8}}\put(13.7,2.5){\makebox(0,0)[cc]{$\bullet$}}
\put(14.2,2.5){\line(1,0){8}}\put(22.7,2.5){\makebox(0,0)[cc]{$\bullet$}}
\put(23.2,2.5){\line(1,0){8}}\put(31.7,2.5){\makebox(0,0)[cc]{$\bullet$}}
\put(36,2.5){\makebox(0,0)[cc]{$\cdots\cdots$}}
\put(5,-4){\makebox(0,0)[cc]{$A$}}\put(2.5,2.5){\makebox(0,0)[cc]{${}_X$}}
\put(13.7,5){\makebox(0,0)[cc]{${}_{\Da_B}$}}
\put(22.7,0){\makebox(0,0)[cc]{${}_{\Da_{B+P}}$}}
\put(31.7,5){\makebox(0,0)[cc]{${}_{\Da_{B+2P}}$}}
\end{picture}
\qquad\qquad
\xygraph{
!{<0cm,0cm>;<.8cm,0cm>:<0cm,.8cm>::}
!{(2,2)}*+{\bullet}="e" !{(2.4,1.6)}*+{{}^{[\Da_{B+(n-1)P}]}}="e1"
!{(4,2)}*+{\bullet}="f" !{(4.4,2.4)}*+{{}^{[\Da_{B+nP}]}}="f1"
!{(6,2)}*+{\bullet}="g" !{(6.4,1.6)}*+{{}^{[\Da_{B+(n-1)P}]}}="g1"
 !{(3,1)}*+{B}="z1"
 "e"-"f" "f"-"g"} 
\]\caption{Two graphs used in the proof of Lemma \ref{l42}. The square marked "X" denotes a possibly infinite subgraph.}\end{figure}
We conclude that the absolute value of the degrees of the divisors $D$ and $D-2P$ 
are different, so that the corresponding orders cannot be conjugate. The result follows.
\end{proof}

\begin{remark}
At the end of the preceding proof, we can also prove that $\Ea$ is not split by observing that 
$\Da_{B+(n+1)P}$, as a neighbor of $\Da_{B+nP}=\Da_\Lambda$, corresponds to a common eigenspace
$V\subseteq\Lambda_P/\pi_P\Lambda_P$ of every idempotent in the ring of global sections $\Da_{B+nP}(X)$, 
whence no such idempotent  has two eigenspaces in $\Lambda_P/\pi_P\Lambda_P$
corresponding to $P$-neighbors isomorphic to $\Da_{B+(n-1)P}$.
\end{remark}

\begin{example}\label{e44}
Let $P$ and $Q$ be points of degree 1 in the proyective line $\mathbb{P}^1$. Let $U=X\backslash\{P\}$.
Consider an order
$\Ea\in\mathbb{O}_{2P}$ and the classifying graph $\mathfrak{c}_Q(\Ea)$. 
The vertices of this graph, or equivalently
the conjugacy classes in $\mathbb{O}_{2P}$, are in correspondence with the 
orbits of lines of length $2$ in the
Bruhat-Tits tree at $P$, under the action of the normalizer of the maximal
$\oink_X(U)$-order $\Ea(U)=\Da(U)$, for any maximal order $\Da\supseteq\Ea$. 
We claim that these orbits correspond precisely 
to lines $\gamma:\mathfrak{i}_{0,2}\rightarrow \mathfrak{c}_P(\Da)$,
 that can be lifted to paths in $\mathfrak{t}(K_P)$. 
The latter condition rules out the maps satisfying $\gamma_V(n_0)=\gamma_V(n_2)=[\Da_{(n+1)P}]$
and $\gamma_V(n_1)=[\Da_{nP}]$, for some $n>0$,
 as such a map has no injective lifting, since $\Da_{nP}$ has a unique
neighbor in the class $[\Da_{(n+1)P}]$. All other simplicial maps 
$\gamma:\mathfrak{i}_{0,2}\rightarrow \mathfrak{c}_P(\Da)$
can be lifted to injective maps in $\mathfrak{t}(K_P)$ (see Fig. 1A),
and hence correspond to conjugacy classes of Eichler orders, provided that the claim holds.
In fact, the $P$-neighbors $\Da'\in[\Da_{(n-1)P}]$ of the order $\Da_{nP}$
correspond to the finite points of the projective line
$\mathbb{P}^1\big(\mathbb{F}(P)\big)$, and the group  $\Da_{nP}(X)^*$  contains upper
triangular matrices that act as arbitrary linear maps on $\mathbb{F}(P)$.
 As this action is 2-transitive, all orders $\Fa_n$ corresponding to lines $\{\gamma,\gamma'\}$
satisfying $\gamma_V(n_0),\gamma_V(n_2)\in[\Da_{(n-1)P}]$ and 
$\gamma_V(n_1)\in[\Da_{nP}]$,
 for a fixed $n$, are conjugates. 
This proves the claim for such classes, and for maps satisfying
$\gamma_V(n_0)=[\Da_{(n-1)P}]$ and $\gamma_V(n_2)=[\Da_{(n+1)P}]$
is even simpler.
 We conclude that all classes in this genus 
are those of the split orders $\Ea[P,P],\Ea[2P,0],\Ea[3P,-P],\dots$ toghether with the classes of the orders
$\Fa_n$ just described. 

The edges of the graph  $\mathfrak{c}_Q(\Ea)$ are in correspondence with the ideal
 grids of the shape shown in Figure 6A. By switching the role played by the places $P$ and $Q$, 
these grids are also in correspondence with lines
$\gamma:\mathfrak{i}_{0,2}\rightarrow \mathfrak{c}_P(\Da)$ in the graph in Figure 2C, where again we
must consider only the maps that can be lifted to lines in the Bruhat-Tits tree. 
A few of these grids are shown in Figure 6C-E. 
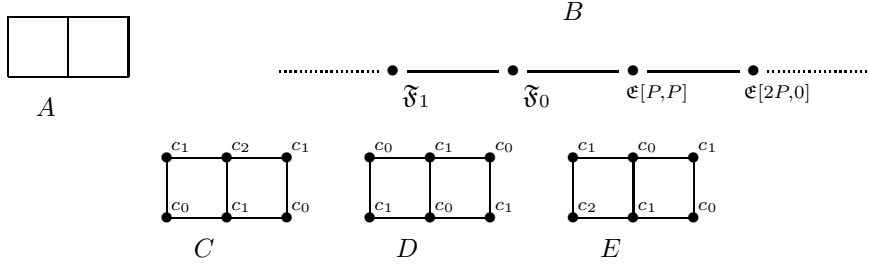
\begin{figure}
\[  
\unitlength 1mm 
\linethickness{0.4pt}
\ifx\plotpoint\undefined\newsavebox{\plotpoint}\fi 
\begin{picture}(20,15)(0,-6)
\put(0,0){\line(0,1){8}}\put(0,8){\line(1,0){16}}\put(0,0){\line(1,0){16}}\put(8,0){\line(0,1){8}}\put(16,0){\line(0,1){8}}
\put(5,-4){\makebox(0,0)[cc]{$A$}}
\end{picture}
\qquad\qquad
 \xygraph{
!{<0cm,0cm>;<.8cm,0cm>:<0cm,.8cm>::}
!{(0,2)}*+{\bullet}="a" !{(0.4,1.6)}*+{\Fa_1}="a1"
!{(2,2)}*+{\bullet}="b" !{(2.4,1.6)}*+{\Fa_0}="b1"
!{(4,2)}*+{\bullet}="c" !{(4.4,1.6)}*+{{}^{\Ea[P,P]}}="c1"
!{(6,2)}*+{\bullet}="x" !{(6.4,1.6)}*+{{}^{\Ea[2P,0]}}="x1"
!{(8,2)}*+{}="d"
!{(-2,2)}*+{}="m"
!{(3,3)}*+{B}="z" 
 "a"-"b" "b"-"c" "c"-"x" "x"-@{.}"d"
 "m"-@{.}"a"}  
\]
\[
\unitlength 1mm 
\linethickness{0.4pt}
\ifx\plotpoint\undefined\newsavebox{\plotpoint}\fi 
\begin{picture}(20,17)(0,-6)
\put(0,0){\line(0,1){8}}\put(0,8){\line(1,0){16}}\put(0,0){\line(1,0){16}}\put(8,0){\line(0,1){8}}\put(16,0){\line(0,1){8}}
\put(0,0){\makebox(0,0)[cc]{$\bullet$}}\put(0,8){\makebox(0,0)[cc]{$\bullet$}}
\put(8,0){\makebox(0,0)[cc]{$\bullet$}}\put(8,8){\makebox(0,0)[cc]{$\bullet$}}
\put(16,0){\makebox(0,0)[cc]{$\bullet$}}\put(16,8){\makebox(0,0)[cc]{$\bullet$}}
\put(2,1){\makebox(0,0)[cc]{${}^{c_0}$}}\put(2,9){\makebox(0,0)[cc]{${}^{c_1}$}}
\put(10,1){\makebox(0,0)[cc]{${}^{c_1}$}}\put(10,9){\makebox(0,0)[cc]{${}^{c_2}$}}
\put(18,1){\makebox(0,0)[cc]{${}^{c_0}$}}\put(18,9){\makebox(0,0)[cc]{${}^{c_1}$}}
\put(5,-4){\makebox(0,0)[cc]{$C$}}
\end{picture}\qquad
\unitlength 1mm 
\linethickness{0.4pt}
\ifx\plotpoint\undefined\newsavebox{\plotpoint}\fi 
\begin{picture}(20,17)(0,-6)
\put(0,0){\line(0,1){8}}\put(0,8){\line(1,0){16}}\put(0,0){\line(1,0){16}}\put(8,0){\line(0,1){8}}\put(16,0){\line(0,1){8}}
\put(0,0){\makebox(0,0)[cc]{$\bullet$}}\put(0,8){\makebox(0,0)[cc]{$\bullet$}}
\put(8,0){\makebox(0,0)[cc]{$\bullet$}}\put(8,8){\makebox(0,0)[cc]{$\bullet$}}
\put(16,0){\makebox(0,0)[cc]{$\bullet$}}\put(16,8){\makebox(0,0)[cc]{$\bullet$}}
\put(2,1){\makebox(0,0)[cc]{${}^{c_1}$}}\put(2,9){\makebox(0,0)[cc]{${}^{c_0}$}}
\put(10,1){\makebox(0,0)[cc]{${}^{c_0}$}}\put(10,9){\makebox(0,0)[cc]{${}^{c_1}$}}
\put(18,1){\makebox(0,0)[cc]{${}^{c_1}$}}\put(18,9){\makebox(0,0)[cc]{${}^{c_0}$}}
\put(5,-4){\makebox(0,0)[cc]{$D$}}
\end{picture}\qquad
\unitlength 1mm 
\linethickness{0.4pt}
\ifx\plotpoint\undefined\newsavebox{\plotpoint}\fi 
\begin{picture}(20,17)(0,-6)
\put(0,0){\line(0,1){8}}\put(0,8){\line(1,0){16}}\put(0,0){\line(1,0){16}}\put(8,0){\line(0,1){8}}\put(16,0){\line(0,1){8}}
\put(0,0){\makebox(0,0)[cc]{$\bullet$}}\put(0,8){\makebox(0,0)[cc]{$\bullet$}}
\put(8,0){\makebox(0,0)[cc]{$\bullet$}}\put(8,8){\makebox(0,0)[cc]{$\bullet$}}
\put(16,0){\makebox(0,0)[cc]{$\bullet$}}\put(16,8){\makebox(0,0)[cc]{$\bullet$}}
\put(2,1){\makebox(0,0)[cc]{${}^{c_2}$}}\put(2,9){\makebox(0,0)[cc]{${}^{c_1}$}}
\put(10,1){\makebox(0,0)[cc]{${}^{c_1}$}}\put(10,9){\makebox(0,0)[cc]{${}^{c_0}$}}
\put(18,1){\makebox(0,0)[cc]{${}^{c_0}$}}\put(18,9){\makebox(0,0)[cc]{${}^{c_1}$}}
\put(5,-4){\makebox(0,0)[cc]{$E$}}
\end{picture}
\]
\caption{The domino-shaped grid (A) used to compute the graph
in Ex. \ref{e44} (B). In (C)-(E) we have the grids corresponding to the three central edges in (B).
Again, we use $c_n=[\Da_{nP}]$.}\end{figure}
We conclude that the graph $\mathfrak{c}_Q(\Ea)$ looks as in Figure 6B. 
\end{example}

\begin{remark}

The previous example is ilustrated in Figure 7 where vertical edges denote $Q$-neighbors while
horizontal edges denote $P$-neighbors. We denote by $\Da_v$ the maximal order corresponding to
a vertex $v$. Assume the vertex denoted $v_0$ corresponds to the 
maximal order $\Da_{v_0}=\mathbb{M}_2(\oink_X)$, and that the frontal plane 
containing the vertices $w$, $v_0$, $z$, $x$, $y$ and $t$ is the infinite grid corresponding to
the cannonical basis, i.e., its vertices correspond precisely to orders split by the cannonical basis.
In analogy with Example \ref{ex42}, we can assume that the Eichler orders corresponding to 
the horizontal lines $\gamma(x,y)$ and $\gamma(w,z)$ are in the same orbit. There exists a 
matrix $M$ in $\mathrm{GL}_2(\mathbb{F})$, the stabilizer of $v_0$, that leaves invariant $z$, 
while sends $w$ to $w'$. However, it can be shown that this matrix does not leave the vertex $x$ 
invariant, mapping the line $\gamma(x,y)$ to a paralell line $\gamma(x',y')$ in a different plane, 
as shown on the right of Figure 7. In fact, if $u'$ denotes the vertex directly below $w'$ in the picture, 
the lines $\gamma(u,t)$ and $\gamma(u',t)$ are in different orbits. 
 In this case the lines above $\gamma(z,w')$ correspond to split orders, while the ones below it are not. 
\begin{figure}
\[
\unitlength 1mm 
\linethickness{0.4pt}
\ifx\plotpoint\undefined\newsavebox{\plotpoint}\fi 
\begin{picture}(34.5,20)(28,6)
\put(30,7){\framebox(28,12)[cc]{}}
\put(37,19){\line(0,-1){12}}
\put(44,19){\line(0,-1){12}}
\put(51,19){\line(0,-1){12}}
\multiput(43.5,19)(.046875,.033653846){208}{\line(1,0){.046875}}
\put(53.25,26){\line(0,-1){6.8}}
\put(48.75,22.7){\line(0,-1){3.7}}
\put(30,13){\line(1,0){28}}
\put(44,19){\makebox(0,0)[cc]{$\bullet$}}\put(43,20){\makebox(0,0)[cc]{${}^{v_0}$}}
\put(37,19){\makebox(0,0)[cc]{$\bullet$}}\put(36,20){\makebox(0,0)[cc]{${}^{w}$}}
\put(51,19){\makebox(0,0)[cc]{$\bullet$}}\put(52,20){\makebox(0,0)[cc]{${}^{z}$}}
\put(37,13){\makebox(0,0)[cc]{$\bullet$}}\put(38,14){\makebox(0,0)[cc]{${}^{u}$}}
\put(44,13){\makebox(0,0)[cc]{$\bullet$}}\put(43,14){\makebox(0,0)[cc]{${}^{x}$}}
\put(58,13){\makebox(0,0)[cc]{$\bullet$}}\put(57,14){\makebox(0,0)[cc]{${}^{y}$}}
\put(51,13){\makebox(0,0)[cc]{$\bullet$}}\put(50,14){\makebox(0,0)[cc]{${}^{t}$}}
\put(48.75,22.7){\makebox(0,0)[cc]{$\bullet$}}\put(47.75,23.7){\makebox(0,0)[cc]{${}^{w'}$}}
\end{picture}
\unitlength 1mm 
\linethickness{0.4pt}
\ifx\plotpoint\undefined\newsavebox{\plotpoint}\fi 
\begin{picture}(34.5,20)(28,6)
\put(30,7){\framebox(28,12)[cc]{}}
\put(37,19){\line(0,-1){12}}
\put(44,19){\line(0,-1){12}}
\put(51,19){\line(0,-1){12}}
\multiput(58,19)(.033653846,-.046875){208}{\line(1,0){.046875}}
\put(58,14.3){\line(1,0){3.5}}
\put(58,9.6){\line(1,0){7}}
\put(30,13){\line(1,0){28}}
\put(44,19){\makebox(0,0)[cc]{$\bullet$}}\put(43,20){\makebox(0,0)[cc]{${}^{v_0}$}}
\put(37,19){\makebox(0,0)[cc]{$\bullet$}}\put(36,20){\makebox(0,0)[cc]{${}^{w}$}}
\put(51,19){\makebox(0,0)[cc]{$\bullet$}}\put(52,20){\makebox(0,0)[cc]{${}^{z}$}}
\put(44,13){\makebox(0,0)[cc]{$\bullet$}}\put(43,14){\makebox(0,0)[cc]{${}^{x}$}}
\put(37,13){\makebox(0,0)[cc]{$\bullet$}}\put(38,14){\makebox(0,0)[cc]{${}^{u}$}}
\put(58,13){\makebox(0,0)[cc]{$\bullet$}}\put(57,14){\makebox(0,0)[cc]{${}^{y}$}}
\put(51,13){\makebox(0,0)[cc]{$\bullet$}}\put(50,14){\makebox(0,0)[cc]{${}^{t}$}}
\put(61.5,14){\makebox(0,0)[cc]{$\bullet$}}\put(63,15){\makebox(0,0)[cc]{${}^{y'}$}}
\end{picture}
\]
\caption{The global orders in Example \ref{e44}.}\end{figure}
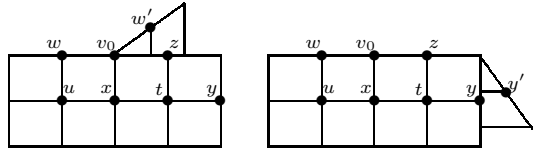
\end{remark}

\subparagraph{Proof of Theorem  \ref{t2}}
Let $D$ be an effective divisor on $X=\mathbb{P}^1$. First we assume $D$ is the sum of at most two different
places of degree $1$. Then $D\leq P_1+P_2$ for some pair of places $P_1$ and $P_2$ of degree 1.
By looking at the product of the local Bruhat-Tits trees at $P_1$ and $P_2$, we observe that 
any order $\Ea\in\mathbb{O}_D$ corresponds to a vertex, edge or grid contained in the concrete
$1$-times-$1$ grid $\mathbb{S}(\Ea')$,
 for an order $\Ea'\in\mathbb{O}_{P_1+P_2}$. The latter is a split order, as shown in Lemma 
\ref{l41}, whence its ring $\Ea'(X)$ of global sections contains a non-trivial idempotent.
Since $\Ea(X)\supseteq\Ea'(X)$, the same holds for $\Ea$, and the result follows. In any other case,
$D\geq B$ for a divisor $B$ in one of the following cases: 
\begin{enumerate}
\item $B=2P$, where $\mathop{\mathrm{deg}}(P)=1$,
\item $B=P_1+P_2+P_3$, where $\mathop{\mathrm{deg}}(P_1)=
\mathop{\mathrm{deg}}(P_2)=\mathop{\mathrm{deg}}(P_3)=1$, or 
\item $B=P$, where $P$ is a place satisfying $\mathop{\mathrm{deg}}(P)>1$.
\end{enumerate}
Then the result follows from Lemma \ref{l42}, Lemma \ref{l41} or Lemma \ref{l34}, respectively,
 by a similar reasoning.
\qed 

\begin{lemma}
Let $X$ be an arbitrary smooth curve, and
let $P_1,\dots,P_n\in |X|$ be different. Then there is only a finite number
of conjugacy classes of non-split orders in  $\mathbb{O}_{P_1+\cdots+P_n}$.
\end{lemma}

\begin{proof}
 This was proved by Serre for the genus of maximal orders, where $n=0$.
 In fact, this follows from  \cite[Th. II.9]{trees}, since by construction all vertices
 in the cusps correspond to split bundles. We finish the proof by induction on $n$.
 Conjugacy classes in $\mathbb{O}_{P_1+\cdots+P_t}$ are in correspondence with the vertices in 
the classifying graph $\mathfrak{c}_{P_{t+1}}(\mathbb{O}_{P_1+\cdots+P_t})$,
 so all but a finite number of them
correspond to the conjugacy class of an order $\Ea[B,B']$, where $B+B'=P_1+\cdots+P_t$. By switching
$B$ and $B'$ if needed, we can assume 
$\mathop{\mathrm{deg}}(B)\leq\mathop{\mathrm{deg}}(B')$.  Furthermore, a second
order $\Ea[B'',B''']\in\mathbb{O}_{P_1+\cdots+P_t}$ with $B''$ linearly equivalent to $B$ is in the same 
conjugacy class, so by leaving out a finite number of conjugacy classes, we can always assume
$\mathop{\mathrm{deg}}(B)<M$ for any prescribed constant $M$. In particular, we can assume also
that $B'$ has positive degree, and therefore $\mathfrak{L}^{-B'}(X)=\{0\}$.
We can further assume that
$$\mathrm{dim}_{\finitum}\Big(\mathfrak{L}^{-B}(X)/\mathfrak{L}^{-B-P_{t+1}}(X)\Big)=
[\finitum(P_{t+1}):\finitum]$$
 by Riemann-Roch' Theorem.  We conclude that $\Ea[B,B'](X)^*$ 
acts on the set of neighbors of $\Ea[B,B']$ with two orbits,
 by Remark \ref{r35}. In particular,
the corresponding vertex on $\mathfrak{c}_{P_{t+1}}(\mathbb{O}_{P_1+\cdots+P_t})$
 has valency one or two, and therefore every ideal grid having the grid corresponding
 to $\Ea[B,B']$ as a $P_{t+1}$-cap corresponds to either of the non-isomorphic bundles 
$\Ea[B+P_{t+1},B']$ or $\Ea[B,B'+P_{t+1}]$,which are both split. As every
ideal $(P_1+\cdots+P_t)$-grid is the $P_{t+1}$-cap of a finite number of ideal
$(P_1+\cdots+P_t+P_{t+1})$-grids, the result follows.
\end{proof}

\subparagraph{Proof of Theorem  \ref{t1}}
 If $D$ a multiplicity-free effective divisor, then the result follows from the preceding lemma. Assume now
that $D$ is not multiplicity-free. Then there is a place $P\in|X|$ satisfying $2P\leq D$. It follows that every
order in $\mathbb{O}_{2P}$ contains an order in $\mathbb{O}_D$. Now the result follows from two 
observations:
\begin{enumerate}
\item Every order containing a split order is split.
\item Every order in $\mathbb{O}_D$ is contained in a finite number of orders in $\mathbb{O}_{2P}$.
\end{enumerate}
The first statement follows since splitting is equivalent to the existence of an idempotent
global section, as in the proof of Theorem \ref{t2}.
The second statement is an immediate consequence of the combinatorial structure of products of 
Bruhat-Tits trees. We conclude from Lemma \ref{l42} that
there is an infinite number of non-conjugate orders in  $\mathbb{O}_D$ contained in non-split orders
in $\mathbb{O}_{2P}$, whence the result follows. \qed
 
\section{Computing fundamental domains for congruence subgroups of $\mathrm{GL}_2(A)$}

In all of this section $A=\oink_X(U)$ for a suitable open set $U=X-\{P_\infty\}$, although later
we specify to the case $A=\mathbb{F}[t]$. Let $K_\infty$ be the completion of $\mathbb{F}(X)$
at $P_\infty$, $\mathcal{O}_{\infty}$ its ring of integers and 
$\nu=\nu_\infty= -\deg$ the valuation function in $K_{\infty}$. 
We identify the Bruhat-Tits tree for $K_\infty$ with the Ball tree, whose vertices are
the closed balls in $K_\infty$, and two of them are neighbors if one is a proper sub-ball of the other.
See \cite[\S4]{omeat} for details. By an end of a graph $\mathfrak{g}$, we mean an equivalence class of rays 
$\rho:\mathfrak{i}_{0,\infty}\rightarrow\mathfrak{g}$, where two rays $\rho$ and $\rho'$  are equivalently precisely
when $\rho_A(a_n)=\rho'_A(a_{n+t})$ for a fixed integer $t$ and every big enough positive integer $n$.
Ends of the Ball tree are naturally in correspondence with the elements of $\mathbb{P}^1(K_\infty)$.
The same holds for its subgraphs. We say that a subgraph $\mathfrak{h}$ contains and end
$a\in\mathbb{P}^1(K_\infty)$ if there is at least one  ray $\rho:\mathfrak{i}_{0,\infty}\rightarrow\mathfrak{h}$ 
in the corresponding equivalence class. We write $a\in\mathfrak{h}$ in this case. As it is the case
for any tree, the Ball tree contains a unique line between any two vertices or ends. The smallest
subtree containing any number of ends and vertices, as the ones mentioned in Theorem \ref{t4},
 is the union of the images of the lines
between each pair of such ends or vertices. 

Recall that quotient graphs are defined here in terms of the baricentric subdivision. In fact, to define
fundamental domains in our context, we perform a surgery on the graph to turn in to a tree. For this,
we choose a maximal tree $\mathfrak{m}$ in the quotient graph having no new half edges, i.e., we remove
some "edges", that in the barycentric subdivision need to be interpreted as path of length $2$ with 
a barycenter in the middle. Each on of these "edges" is replace by a pair of half edges, and the same is done at every
preimage in $\mathfrak{t}(P_\infty)$ of such vertices. Then any lifting of the resulting tree
to the barycentric subdivision of $\mathfrak{t}(P_\infty)$ is called
a fundamental domain. Note that the quotient graph can be recover from
the fundamental domain and the pairs of nonvertices that must be glued. See Fig. \ref{newfigure}.

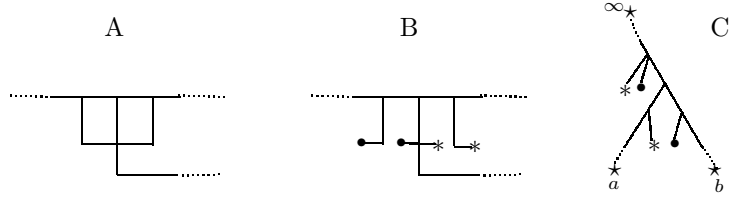
\begin{figure}
\[
\unitlength .7mm 
\linethickness{0.4pt}
\ifx\plotpoint\undefined\newsavebox{\plotpoint}\fi 
\begin{picture}(139.75,33.5)(0,0)
\put(16,22.75){\line(1,0){24.5}}
\put(21.75,22.5){\line(0,-1){8.75}}
\put(21.75,13.75){\line(1,0){13.5}}
\put(35.25,13.5){\line(0,1){9.25}}
\put(28.5,22.5){\line(0,-1){14.5}}
\put(28.5,8){\line(1,0){11.5}}
\multiput(40.18,22.93)(.94444,-.02778){10}{{\rule{.4pt}{.4pt}}}
\multiput(15.93,22.68)(-.96875,.03125){9}{{\rule{.4pt}{.4pt}}}
\multiput(40.93,7.68)(.96429,.03571){8}{{\rule{.4pt}{.4pt}}}
\put(73.25,22.75){\line(1,0){24.5}}
\put(79,22.5){\line(0,-1){8.75}}
\put(92.5,13.5){\line(0,1){9.25}}
\put(85.75,22.5){\line(0,-1){14.5}}
\put(85.75,8){\line(1,0){11.5}}
\multiput(97.43,22.93)(.94444,-.02778){10}{{\rule{.4pt}{.4pt}}}
\multiput(73.18,22.68)(-.96875,.03125){9}{{\rule{.4pt}{.4pt}}}
\multiput(98.18,7.68)(.96429,.03571){8}{{\rule{.4pt}{.4pt}}}
\put(92.5,13.25){\line(1,0){3.25}}
\multiput(89,13.75)(-.84375,.03125){8}{\line(-1,0){.84375}}
\put(79,14){\line(-1,0){4.5}}
\multiput(127.75,33.5)(.0337078652,-.0575842697){356}{\line(0,-1){.0575842697}}
\multiput(132.5,25.25)(-.033632287,-.052690583){223}{\line(0,-1){.052690583}}
\multiput(129.5,20.25)(.0333333,-.3833333){15}{\line(0,-1){.3833333}}
\multiput(135.75,20)(-.03333333,-.11666667){45}{\line(0,-1){.11666667}}
\multiput(129.25,31)(-.033613445,-.048319328){119}{\line(0,-1){.048319328}}
\multiput(129.5,30.75)(-.03333333,-.11666667){45}{\line(0,-1){.11666667}}
\multiput(127.68,33.43)(-.35,.75){6}{{\rule{.4pt}{.4pt}}}
\multiput(139.68,13.18)(.41667,-.66667){7}{{\rule{.4pt}{.4pt}}}
\multiput(125.18,13.93)(-.45,-.75){6}{{\rule{.4pt}{.4pt}}}
\put(75,14){\makebox(0,0)[cc]{${}_\bullet$}}\put(82.6,14){\makebox(0,0)[cc]{${}_\bullet$}}
\put(89.5,13.5){\makebox(0,0)[cc]{$*$}}\put(96.5,13.3){\makebox(0,0)[cc]{$*$}}
\put(125,24){\makebox(0,0)[cc]{$*$}}\put(128.2,24.5){\makebox(0,0)[cc]{${}_\bullet$}}
\put(130.5,13.5){\makebox(0,0)[cc]{$*$}}\put(134.5,13.8){\makebox(0,0)[cc]{${}_\bullet$}}
\put(126,39){\makebox(0,0)[cc]{$\star$}}\put(123,9){\makebox(0,0)[cc]{$\star$}}
\put(123,39){\makebox(0,0)[cc]{${}^\infty$}}\put(123,6){\makebox(0,0)[cc]{${}_a$}}
\put(143,6){\makebox(0,0)[cc]{${}_b$}}\put(142,9){\makebox(0,0)[cc]{$\star$}}
\put(28,36){\makebox(0,0)[cc]{A}}\put(84,36){\makebox(0,0)[cc]{B}}
\put(143,36){\makebox(0,0)[cc]{C}}
\end{picture}
\]
\caption{Surgery of the quotient graph (A, B) and one choice of a fundamental domain (C).
Bullets and asterisks denote two corresponding pairs of nonvertices. Stars denote ends.}
\label{newfigure}
\end{figure}

\begin{example}\label{e51}
Assume $A=\mathbb{F}[t]$. 
In Figure 8 we can see the minimal subgraph $\mathfrak{s}$ containing
$0$, $\infty$ and each $M^{-1}$ with $M$ dividing $N$ for $N=t(t-1)$ or $N=t(t-1)(t-2)$.
In the latter case we assume $\mathrm{char}(\mathbb{F})>2$.
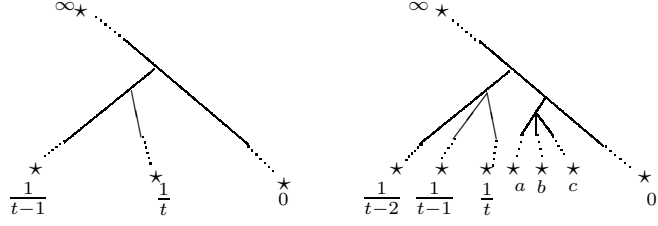
\begin{figure}
\[
\unitlength 1mm 
\linethickness{0.4pt}
\ifx\plotpoint\undefined\newsavebox{\plotpoint}\fi 
\begin{picture}(73.5,27.25)(0,0)
\multiput(19.75,27.25)(.0397590361,-.0337349398){415}{\line(1,0){.0397590361}}
\multiput(23.75,23.25)(-.0415224913,-.0337370242){289}{\line(-1,0){.0415224913}}
\put(20.75,20.5){\line(1,-5){1.25}}
\multiput(67,27)(.0397590361,-.0337349398){415}{\line(1,0){.0397590361}}
\multiput(71,23)(-.0415224913,-.0337370242){289}{\line(-1,0){.0415224913}}
\put(68,20.25){\line(1,-5){1.25}}
\put(68,20){\line(-3,-4){4.5}}
\multiput(75.75,19.5)(-.033505155,-.051546392){97}{\line(0,-1){.051546392}}
\multiput(74.5,17.5)(.03358209,-.04477612){67}{\line(0,-1){.04477612}}
\put(74.5,17.75){\line(0,-1){3.25}}
\multiput(66.93,26.68)(-.58333,.5){7}{{\rule{.4pt}{.4pt}}}
\multiput(63.68,13.93)(-.4375,-.625){5}{{\rule{.4pt}{.4pt}}}
\multiput(69.18,13.68)(-.125,-.8125){5}{{\rule{.4pt}{.4pt}}}
\multiput(74.68,14.43)(.1667,-.9167){4}{{\rule{.4pt}{.4pt}}}
\multiput(76.68,14.18)(.375,-.5625){5}{{\rule{.4pt}{.4pt}}}
\multiput(19.68,27.18)(-.625,.5){7}{{\rule{.4pt}{.4pt}}}
\multiput(11.68,13.68)(-.5,-.4){6}{{\rule{.4pt}{.4pt}}}
\multiput(22.18,14.18)(.2,-.7){6}{{\rule{.4pt}{.4pt}}}
\multiput(36.18,12.93)(.54167,-.54167){7}{{\rule{.4pt}{.4pt}}}
\multiput(59.18,13.18)(-.4,-.4){6}{{\rule{.4pt}{.4pt}}}
\multiput(72.68,14.43)(-.25,-.9167){4}{{\rule{.4pt}{.4pt}}}
\multiput(83.68,12.93)(.65,-.55){6}{{\rule{.4pt}{.4pt}}}
\put(74.5,17.75){\line(0,-1){3.25}}
\put(41,8){\makebox(0,0)[cc]{$\star$}}\put(24,9){\makebox(0,0)[cc]{$\star$}}
\put(41,6){\makebox(0,0)[cc]{${}_0$}}\put(25,6){\makebox(0,0)[cc]{$\frac1t$}}
\put(8,10){\makebox(0,0)[cc]{$\star$}}\put(14,31){\makebox(0,0)[cc]{$\star$}}
\put(7,6){\makebox(0,0)[cc]{$\frac1{t-1}$}}\put(12,31){\makebox(0,0)[cc]{${}^\infty$}}
\put(89,9){\makebox(0,0)[cc]{$\star$}}\put(68,10){\makebox(0,0)[cc]{$\star$}}
\put(90,6){\makebox(0,0)[cc]{${}_0$}}\put(68,6){\makebox(0,0)[cc]{$\frac1t$}}
\put(62,10){\makebox(0,0)[cc]{$\star$}}\put(56,10){\makebox(0,0)[cc]{$\star$}}
\put(61,6){\makebox(0,0)[cc]{$\frac1{t-1}$}}\put(54,6){\makebox(0,0)[cc]{$\frac1{t-2}$}}
\put(62,31){\makebox(0,0)[cc]{$\star$}}
\put(59,31){\makebox(0,0)[cc]{${}^\infty$}}
\put(71.5,10){\makebox(0,0)[cc]{$\star$}}\put(75.3,10){\makebox(0,0)[cc]{$\star$}}
\put(72.5,7){\makebox(0,0)[cc]{${}^a$}}\put(75.3,7){\makebox(0,0)[cc]{${}^b$}}
\put(79.5,7){\makebox(0,0)[cc]{${}^c$}}\put(79.5,10){\makebox(0,0)[cc]{$\star$}}
\end{picture}
\]
\caption{The global orders in Example \ref{e51}. Here, $a=\frac1{t(t-1)}$ $b=\frac1{t(t-2)}$
and  $c=\frac1{(t-1)(t-2)}$.}\end{figure}
\end{example}

\paragraph{Proof of Theorem \ref{t4}}

Set $P_i\in|\mathbb{P}^1_{\mathbb{F}}|$ to be the point corresponding to $\lambda_i$, or equivalently
 assume $\mathrm{div}(x-\lambda_i)=P_i-P$, where $P=P_\infty$ denotes the place at infinity.
Repetitive use of Example \ref{ex42} shows that the classifying graph 
$\mathfrak{c}_P(\mathbb{O}_{P_1+\dots+P_n})$ has a unique cusp. The natural cover 
$\psi:\mathfrak{s}_P(\mathbb{O}_{P_1+\dots+P_n})\twoheadrightarrow 
\mathfrak{c}_P(\mathbb{O}_{P_1+\dots+P_n})$
is at most $2^n$-to-one, as $2^n$ is the order of the group $\Gamma/\Gamma_0(\mathfrak{E})$, in the notations
of Remark \ref{r35} for $\mathfrak{E}\in\mathbb{O}_{P_1+\dots+P_n}$,  by
\cite[Th. 1.2]{scffgeo}. Note that $\Gamma_0(\mathfrak{E})=\Gamma_N/K^*$.
It suffices, therefore, to prove that the restriction of $\psi$ to the tree $\mathfrak{s}$
is an injection. Consequently, the result follows from next result:

\begin{lemma}
The vertices in $\mathfrak{s}$ are in different $\Gamma_N$-orbits.
\end{lemma}

\begin{proof}
Note that the lemma is well known if $N=1$ is maximal, so we assume troughout that this 
is not the case. We use $B_x^{|t|}$ for the ball of radius $|\pi|^t$ centered at $x\in K_\infty$, where
$\pi=\pi_{P_\infty}$ is a uniformizing parameter. Set $B_0=B_0^{|0|}$, the ball corresponding to $\Da_0$.
  Let $B_1=B_{x_1}^{|r_1|}$ and $B_2=B_{x_2}^{|r_2|}$ two vertices in $\mathfrak{s}$, 
where $x_1,x_2$ are $0$ or the inverse 
of a proper monic divisor of $N$. Assume that there exists a matrix $g=\sbmattrix ab{Nc}d\in\Gamma_N$
 satisfying $g.B_1=B_2$. Let 
Set $h_1= \sbmattrix {x_1}{\pi^{r_1}}{1}{0}$ and $h_2=\sbmattrix {x_2}{\pi^{r_2}}{1}{0}$,
so that $B_1=h_1.B_0$ and $B_2=h_2.B_0$.
Then, for some $\lambda \in K_{\infty}^{*}$, we must have 
$h_2^{-1}gh_1 \in \lambda \mathrm{GL}_2(\mathcal{O}_{\infty})$, as 
$K_\infty\mathrm{GL}_2(\mathcal{O}_{\infty})$ is 
the stabilicer of $B_0$. By taking determinants,
we get $2\nu(\lambda)=r_1-r_2$. Hence, $r_1-r_2$ is an even integer and 
$\pi^{\frac{r_2-r_1}{2}}h_2^{-1}gh_1 \in\mathrm{GL}_2(\mathcal{O}_{\infty})$. 
After a simple computation we have
\begin{equation}\label{eq def}
\bmattrix {\pi^{\frac{r_2-r_1}2}(d+Ncx_1)}{\pi^{\frac{r_2+r_1}2}Nc}
{\pi^{\frac{-r_1-r_2}2}(ax_1-dx_2+b-Ncx_1x_2)}{\pi^{\frac{r_1-r_2}2}(a-Ncx_2)}
\in\mathrm{GL}_2(\mathcal{O}_{\infty}).
\end{equation}
 We conclude that
$\pi^{\frac{r_1-r_2}{2}}(a-Ncx_2),\pi^{\frac{r_2-r_1}{2}}(d+Ncx_1) \in \mathcal{O}_{\infty}$. 
On the other hand, the polynomials $a-Ncx_2$ and $d+Ncx_1$ either vanish or have non-positive valuations.
This leaves three alternatives:
\begin{itemize}
\item[(i)] $r:=r_1=r_2$, toghether with $\nu(a-Ncx_2)=\nu(d+Ncx_1)=0$,
\item[(ii)] $a=Ncx_2$ or 
\item[(iii)] $d=-Ncx_1$.
\end{itemize}
 The last two alternatives imply $\det(g) \notin \mathbb{F}^{*}$, so (i) must hold.
The result follows if $x_1=x_2$, as this implies both balls are identical. We assume
in the sequel that $x_1\neq x_2$. 
From \eqref{eq def} and (i) we deduce the following facts:
\begin{itemize}
\item[(a)] $a-Ncx_2 = a_0 \in \mathbb{F}^*$,
\item[(b)] $d+Ncx_1 =d_0 \in \mathbb{F}^*$,
\item[(c)] $ Nc \in \pi^{-r}\mathcal{O}_{\infty}$, or equivalently $\deg(Nc)\leq r$, 
so in particular $r>0$, and
\item[(d)] $a_0 x_1- d_0 x_2 + b+Ncx_1x_2 = ax_1-dx_2+b-Ncx_1x_2 \in \pi^r \mathcal{O}_{\infty}$.
\end{itemize}

Note that $x_1$ and $x_2$ do not vanish simultaneously by the previous assumption.
 If we suppose that either $\nu(Ncx_1x_2)>0$ or $x_1x_2=0$,
 then the dominant term in the left hand side of identity of (d) is 
$b \in \mathbb{F}[t]$, unless it vanishes. As $r>0$ we must conclude the latter.
 It follows that $g=\sbmattrix {a}{0}{Nc}{d}$, in particular $a,d \in \mathbb{F}^{*}$. 
This can only mean $Nc x_2, Ncx_1 \in \mathbb{F}$, 
and then $c=0$, as at least one element in $\{x_1,x_2\}$ is the inverse of an proper monic divisor of $N$.
 From the preceeding considerations, we get the identity $B_2=g.B_1= B_{ax_1/d}^{|r|}$, 
whence $a=d$ and $B_1=B_2$.

Finally, assume that both $x_1, x_2 \neq 0$ and $\nu(Ncx_1x_2)\leq 0$. We can assume $r> \max{ \lbrace \nu(x_1), \nu(x_2)\rbrace }$  or we could redefine $x_1$ or $x_2$ by $0$ and return to the preceding case. Let
 \begin{equation}\label{bin}\epsilon= b+Ncx_1x_2 \in a_0x_1+d_0x_2+\pi^r\oink_\infty\subseteq
\pi \mathcal{O}_{\infty}.\end{equation}
 By a simple computation we get $\det(g) =a_0 d_0- \xi \in \mathbb{F}^{*}$, 
where $ \xi= Nc(a_0 x_1-d_0 x_2 + \epsilon)\in \mathbb{F}$.  If $\xi=0$, we have that $c=0$ or
\begin{equation}\label{eq divisi}
Nc+ bx_1^{-1} x_2^{-1} =\epsilon (x_1x_2)^{-1}= d_0 x_1^{-1}-a_0x_2^{-1}.
\end{equation} 
In the former case $b \in \pi\mathcal{O}_{\infty}$ by \eqref{bin}, so that $b=0$ and we argue as in the previous paragraph. In the latter case, equation  \eqref{eq divisi} implies that $x_1^{-1}$ divides to $x_2^{-1}$ 
and inversely, as each divides $N$, whence $B_1=B_2$.

Assume now that $\xi \neq 0$, so by (c) and (d) we get
$$r\geq-\nu(Nc)=\nu(a_0 x_1- d_0 x_2 +\epsilon)=\nu(a_0 x_1- d_0 x_2 +b+Ncx_1x_2)\geq r,$$
whence $\nu(a_0 x_1- d_0 x_2 +\epsilon)=-\nu(Nc)= r$. In this case we have
$$
|\pi^r|=|a_0 x_1- d_0 x_2 + \epsilon|=|x_1x_2||a_0x_2^{-1}-d_0x_1^{-1}+ \epsilon(x_1x_2)^{-1}| \geq |x_1x_2|,
$$
as the second factor is a polynomial.
On the other hand, the hypothesis $\nu(Ncx_1x_2)\leq 0$ implies
$|x_1x_2| = |Ncx_1 x_2||Nc|^{-1}\geq|\pi^r|$.
Thus, $r=\nu(x_1x_2)$ and $\sigma=a_0x_2^{-1}-d_0x_1^{-1}+ b(x_1x_2)^{-1} + Nc$ is a non zero constant polynomial.
 But $\sigma$ is divisible by $\text{gcd}(x_1^{-1},x_2^{-1})$, and therefore $\text{gcd}(x_1^{-1},x_2^{-1})=1$.
 If $\epsilon \neq 0$ we conclude that $b(x_1x_2)^{-1} + Nc$ is a multiple of $(x_1x_2)^{-1}$. 
By the strong triangular inequality, $|\sigma|=1$ impplies
$$|a_0x_1^{-1}-d_0x_2^{-1}| = |b(x_1x_2)^{-1} + Nc|\geq |x_1x_2|^{-1}.$$
The preceeding inequality is impossible by a degree argument.
To finish the proof we consider $\epsilon=0$, in which case $|a_0x_1^{-1}-d_0x_2^{-1}|=1$ by (5).
 As the polynomials are monic, this is only possible
 when $a_0=d_0$ and $|x_1-x_2|\leq|\pi^r|$. We conclude that $B_1=B_2$. 
\end{proof}

\section{Acknowledgements}
The first author was suported by Fondecyt-Conicyt, Grant No 1180471. The second author was supported by Conycyt,
Doctoral fellowship No $21180544$.

$$ $$

{\sc Luis Arenas-Carmona}\\
  Universidad de Chile, Facultad de Ciencias, \\Casilla 653, Santiago, Chile\\
  \email{learenas@u.uchile.cl}

$$ $$

{\sc Claudio Bravo}\\
Departamento de Matem\'atica,\\
Universidad T\'ecnica Federico Santa Mar\'ia, \\
Avda. Espa\~na 1680, Valpara\'iso, Chile\\
  \email{claudio.bravoc@usm.cl}

\end{document}